\documentclass[11pt,a4paper]{article}
\ifdefined\drft
\usepackage[top=3cm, bottom=3cm, right=5cm, left=1cm]{geometry}
\usepackage{setspace}
\else
\usepackage[top=3cm, bottom=2.5cm, right=2.5cm, left=2.5cm]{geometry}
\fi

\usepackage[]{hyperref} 

\usepackage{ucs}                
\usepackage[utf8]{inputenc} 	
\usepackage[T1]{fontenc}     
\usepackage{lmodern}

\usepackage[sfdefault=cmbr,OMLmathsans]{isomath}
\usepackage{upgreek}  
\usepackage{tensor}  

\usepackage{amsmath}
\usepackage{amsthm}
\usepackage{amssymb}
\usepackage{mathrsfs}
\usepackage{amsfonts}
\usepackage{mathtools} 
\usepackage{bbm} 

\usepackage{algorithm}
\usepackage{algorithmicx}
\usepackage{algpseudocode}
\usepackage{graphicx}
\usepackage[normalsize]{subfigure}
\usepackage{xifthen}
\usepackage{rotating}
\usepackage{xcolor}
\usepackage{tikz-cd}

\newcommand{\ass}{\mathrm{a}}
\newcommand{\cov}{\mathrm{cov}}
\providecommand{\F}{\mathfrak}

\newcommand{\CEsign}{\Phi}
\newcommand{\CE}[2]{\CEsign_{#1|#2}}
\newcommand{\aCE}[2]{\widetilde{\CEsign}_{#1|#2}}

\newcommand{\YE}{Z}
\newcommand{\cy}{{\hat{y}}}
\newcommand{\sY}{\mathcal{Y}}
\newcommand{\q}{q}
\newcommand{\Q}{Q}

\newcommand{\ch}[1]{\mathbbm{1}_{\!#1}} 
\newcommand{\dual}[2]{\langle #1,#2 \rangle}
\newcommand{\figscalepar}{0.45}
\newcommand{\Yq}{Y_Q}
\newcommand{\Z}{Z}

\newcommand{\distN}{\mathcal{N}}

\newcommand{\sep}{\,|\,}

\newcommand{\E}[1]{\mathbb{E}\bigl[#1\bigr]}
\renewcommand{\P}{\mathbb{P}}

\newcommand{\sQ}{\mathcal{Q}}
\newcommand{\sL}{\mathrm{L}}
\newcommand{\sX}{\mathcal{X}}

\newcommand{\sB}{\F{B}}
\newcommand{\Adel}{A_{\del}}
\newcommand{\Ay}{A_{\cy}}

\newcommand{\skipifemptyarg}[1]{\ifthenelse{\isempty{#1}}{}{\left[#1\right]}}
\newcommand{\skipifscalar}[1]{\ifthenelse{\isempty{#1}}{}{;#1}}

\newcommand{\bs}[1]{\boldsymbol{#1}}
\newcommand{\scal}[2]{\bigl(#1,#2\bigr)}
\newcommand{\set}[1]{\mathbb{#1}} 
\newcommand{\V}[1]{\bs{#1}} 
\newcommand{\ol}[1]{\overline{#1}}

\newcommand{\alp}{\ensuremath{\alpha}}

\newcommand{\del}{\ensuremath{\delta}}

\newcommand{\sR}{\set{R}}

\DeclareMathOperator*{\argmin}{arg\,min}

\newcommand{\meas}[1]{\lambda(#1)}

\newcommand{\D}[1]{\,{\mathrm{d}}#1}

\theoremstyle{plain}
\newtheorem{theorem}{Theorem}
\newtheorem{definition}[theorem]{Definition}
\newtheorem{lemma}[theorem]{Lemma}
\newtheorem{corollary}[theorem]{Corollary}
\newtheorem{remark}[theorem]{Remark}

\newtheorem{proposition}[theorem]{Proposition}
\newtheorem{problem}[theorem]{Problem}
\newtheorem{example}[theorem]{Example}

\title{Accurate computation of conditional expectation for highly non-linear problems
}
\usepackage{authblk}
\author[1]{Jaroslav~Vondřejc}
\author[1]{Hermann G. Matthies}
\affil[1]{Technische Universit\"{a}t Braunschweig, Institute of Scientific Computing, Mühlenpfordstra{\ss}e~23, 38106 Braunschweig, Germany, \href{mailto:j.vondrejc@tu-bs.de}{j.vondrejc@tu-bs.de}, \href{mailto:vondrejc@gmail.com}{vondrejc@gmail.com}, \href{mailto:wire@tu-bs.de}{wire@tu-bs.de}}
\date{\today}

\begin{document}
\maketitle

\begin{abstract}
This paper focuses on inverse problems to identify parameters by incorporating information from measurements.
These generally ill-posed problems are formulated here in a probabilistic setting based on Bayes's theorem because it leads to a unique solution of the updated distribution of parameters.
Many approaches build on Bayesian updating in terms of probability measures or their densities.
However, the uncertainty propagation problems and their discretisation within the stochastic Galerkin or collocation method are naturally formulated for random vectors which calls for updating of random variables, i.e.\ a filter.
Such filters typically build on some approximation to conditional expectation (CE).
Specifically, the approximation of the CE with affine functions leads to the familiar Kalman filter which works best on linear or close to linear problems only.
Our approach builds on a reformulation, which allows to localise the operator of the CE to the point of measured value.
The resulting conditioned expectation (CdE) predicts correctly the quantities of interest, e.g.\ conditioned mean and covariance, even for general highly non-linear problems.
The novel CdE allows straight-forward numerical integration; particularly, the approximated covariance matrix is always positive definite for integration rules with positive weights.
The theoretical results are confirmed by numerical examples.
\end{abstract}
\tableofcontents
\section{Introduction}
\emph{Inverse problems.}
This paper is focused on the identification of parameters, given by a random variable $\Q:\Omega\rightarrow\sQ$, using measurements. This problem has an enormous importance in engineering practice; unfortunately, the parameters usually cannot be observed directly but only indirectly as a response of some system, given by an \emph{observation operator}
$\Yq:\sQ\rightarrow\sY$. It is therefore called an \emph{inverse problem} \cite{Tarantola2005} contrary to the forward problem $\Yq$.
Moreover, the measurements are also polluted by some error distributed in accordance to a random variable $E:\Omega\rightarrow\sY$. The \emph{measured value} $\cy\in\sY$ is then the sample of observation operator
\begin{align}
\label{eq:inverse_problem}
\cy = \YE(\omega) =\Yq(Q(\omega)) + E(\omega),
\end{align}
where we assume an additive error; for concrete problem see Example~\ref{ex:1}.
Inverse problems are usually ill-posed because an observation can lead to many parameters satisfying \eqref{eq:inverse_problem}, the parameters $q$ can be very sensitive to observation $\cy$, or even the observation $\cy$ can be out of the range of the observation operator $\Yq$.
This mathematically manifests itself in the fact that usually $\Yq$ is not invertible.

\emph{Bayesian updating.}
The probabilistic description of inverse problems is an elegant approach that relies on Bayes's theorem \cite{Ernst2014,Rao2006book,Stuart2010}.
Modelling the \emph{prior} parameters $q=Q(\omega)$ as uncertain, the ill-posed inverse problem turns out to be well-posed with a unique (stochastic) solution.
Therefore the \emph{updated (or assimilated)} parameters can describe all possible parameters (along with their probabilities) satisfying the inverse problem \eqref{eq:inverse_problem} in a stochastic sense.
Many Bayesian approaches to inverse problems employ the updating in terms of probability measures or corresponding densities. 
However, the response of the system (forward problem), from which the measurements are taken, is typically expressed as a random variable (RV), which makes repeated updating with probability densities a non-straightforward computation.

\emph{Conditional expectation.}
The Bayesian updating in terms of random variables builds on the concept of \emph{conditional expectation} (CE), which is defined as another RV.
However, its approximation is a difficult task.
The approximation of a CE with affine (linear) mappings leads to the familiar \emph{Kalman filter} (KF) \cite{Kalman1960}. There are also many sampling variants such as the Ensemble KF \cite{Evensen1994,Ernst2015}, Extended KF \cite{Julier2004}, or accelerated KF employing a surrogate model for the forward problem \cite{Li2009}.
There is also a sampling-free variant of the filter called polynomial chaos Gauss-Markov-Kalman filter \cite{Blanchard2010,Pajonk2012,Rosic2013,Rosic2012} when all random variables are expressed in terms of polynomial chaos expansion (PCE) and only their coefficients are updated.
However, all the variants provide reliable results that match the conditional distribution only for linear or close to linear forward problems.
The distribution of the filtered RV is duly exact for Gaussian linear problems, and deviates from it more and more as this condition is not met.
A better outcome can be achieved with a discretisation of the CE by higher order polynomials \cite{Matthies2015,Matthies2016}, however its approximation properties are still poor for general non-linear problems, see Section~\ref{sec:filtering} for details.

\emph{Novelty.}
The conditional expectation $\CE{\Q}{\YE}$ is a map (from measurement to parameter space), which is generally very difficult to approximate.
For calculating the posterior quantities one needs only a value of the CE at the point of observation $\cy$.
Here we focused on the novel evaluation of this value $\CE{\Q}{\YE}(\cy)$ after the observation $\cy$, called here conditioned expectation (CdE), using the variational formulation of CE.
It overcomes the limitations of previous approaches and leads to an accurate and more robust method presented in Section~\ref{sec:simple-conditional-expectation-of-random-fields}. 

\emph{Organisation of the paper.}
The target of this paper is to provide, besides new results, a variational view on CE from the computational perspective.
Because of various notations and approaches to CE, this paper is designed to be self-contained, including many proofs or their outlines that can help in understanding.
Particularly the existing theory about CE, conditional probabilities, and conditional probability densities is described in Appendix~\ref{sec:conditional-expectation}.
Then the main text contains a brief version in Section~\ref{sec:inverse_via_Bayes} which is accompanied by a model inverse problem in a Bayesian setting and also by existing filters, e.g.\ the Kalman filter.
The novel development and numerical approximation of conditioned expectation is covered in Section~\ref{sec:simple-conditional-expectation-of-random-fields}.
The results are confirmed by numerical examples in Section~\ref{sec:numerical-examples}, and their summary can be found in the conclusion in Section~\ref{sec:conclusion}.

\section{Inverse problems via Bayesian updating}
\label{sec:inverse_via_Bayes}
This section is focused on inverse problems using Bayesian updating.
The basic notation and definitions of random variables and probability distributions in Section~\ref{sec:random-variables-and-probability-distributions} is followed by the description of a model problem in Section~\ref{sec:model-problem}.
Bayesian updating for probability density functions (PDF) and conditional expectation is summarised in Sections \ref{sec:bayes-theorem} and~\ref{sec:conditional-expectation_short};
the full version of this topic is covered in Appendix~\ref{sec:conditional-expectation}.
The last Section~\ref{sec:filtering} of this part is devoted to Bayesian filters along with critical discussion about its approximation.

\subsection{Random variables and probability distributions}\label{sec:random-variables-and-probability-distributions}
The definitions and notations about random variables and probability distributions are introduced here.
Let $(\Omega,\F{S},\P)$ be a probability space with sample space $\Omega$, set of events $\F{S}$ as a $\sigma$-algebra, and probability measure $\P:\F{S}\rightarrow[0,1]$.
The parameter space $\sQ$ (and also the measurement space $\sY$) will be assumed to be a finite dimensional vector space with a scalar product $\scal{u}{v}_\sQ$ or $u\cdot v$ and induced norm $\|u\|_\sQ=\sqrt{u\cdot u}$.
It is considered to be a measurable space $(\sQ,\sB_\sQ)$ with Borel $\sigma$-algebra $\sB_\sQ$ (or simply $\sB$) generated by the open sets in $\sQ$.

A measurable function $X:\Omega\rightarrow\sQ$ is called a random variable (RV). 
The function space $\sL^p(\Omega,\F{S},\P;\sQ)$ (or simply $\sL^p(\Omega;\sQ)$ or $\sL^p$) is the collection of (equivalence classes of) RVs with finite norm $\|X\|_{\sL^p}^p=\int_\Omega{\|X(\omega)\|_Q^p}<\infty$. The bilinear form $\dual{\cdot}{\cdot}_{L^p\times L^q}$ on $L^p\times L^q$ for $p\in[1,\infty)$ and $q=\frac{p}{p-1}$ is defined as a duality pairing
$\dual{f}{g}_{L^p\times L^q}=\int_\Omega f(\omega) g(\omega)\P(\omega)$, and it becomes a scalar product $\dual{f}{g}_{L^2\times L^2}=\scal{f}{g}_{L^2}$ for $p=2$, and $L^2$ is a Hilbert space.
The expectation $\mathbb{E}:\sL^p(\Omega,\F{S},\P;\sQ)\rightarrow\sQ$ is denoted as
$\E{X} = \ol{X} = \int_\Omega X(\omega) \P(\D{\omega})$.
The characteristic function $\ch{A}$ of a set $A\subset\Omega$ is a function that equals $1$ on $A$ and zero otherwise.
The composition of two random variables $Y\circ X$ is defined as $Y\circ X(\omega)=Y[X(\omega)]$.

\subsection{The model problem}
\label{sec:model-problem}
The model problem focuses on the identification of parameters $\q\in\sQ$, which are initially described with a distribution determined by a random vector $\Q\in \sL^1(\Xi,\F{S}_\Xi,\P_\Xi;\sQ)$ with $\q=\Q(\xi)$, or by a probability density function (PDF) $f_\Q$.
Unfortunately, the parameters $\q$ can usually not be observed directly but only a response of it given by the \emph{observation operator} $\Yq:\sQ\rightarrow\sY$.
Here the observation space is typically a finite dimensional vector space $\sR^m$, which corresponds to the information obtained from $m$ measurements.
The \emph{measurement} operator thus generates a measurement random vector $Y\in \sL^1(\Xi,\F{S}_\Xi,\P_\Xi;\sY)$ defined as
\begin{align}
\label{eq:observation}
Y(\xi) = \Yq(\Q(\xi)).
\end{align}
However, the measurement is typically influenced by a \emph{measurement error} $e=E(\theta)\in\sY$ with a distribution given by a random vector $E\in\sL^1(\Theta,\F{S}_\Theta,\P_\Theta;\sY)$ or probability density $f_E$, which has usually zero-mean and a Gaussian distribution.
The measurement \eqref{eq:observation} accompanied by the error establishes a random variable (\emph{measurement with error})
\begin{align}
\label{eq:measurement}
\YE(\xi,\theta) &=\Yq[Q(\xi)]+E(\theta) = Y(\xi)+E(\theta),
\end{align}
form the space $\sL^1(\Omega,\F{S},\P;\sY)$ with the product space $\Omega=\Xi\times\Theta$, product $\sigma$-algebra $\F{S}=\F{S}_\Xi\times\F{S}_\Theta$, and product measure $\P=\P_\Xi\times\P_\Theta$, which means that the error $E$ is independent of the parameter $\Q$. For reasons of simplicity we talk about $\YE$ as the \emph{measurement}.

The random variable $\Q$, defining the distribution of the parameters, is primarily defined on $\Xi$. We will inject that to $\Omega=\Xi\times\Theta$ by defining
$\hat{\Q}(\xi,\theta)=\Q(\xi)$; for notational simplicity we will use only one symbol $\Q$, so $\Q\in\sL^1(\Omega,\F{S},\P;\sQ)$ as well as $\Q\in\sL^1(\Xi,\F{S}_\Xi,\P_\Xi;\sQ)$.
The structure of the measure spaces and maps is summarised in the following diagram.
\begin{equation}
\label{diag:commutative_diag}
\begin{tikzcd}[column sep=5.em,row sep=3.em,every arrow/.append style={shift left}]
(\Omega,\F{S},\P)=\overbrace{(\Xi,\F{S}_\Xi,\P_\Xi)}^{\text{parameter dom.}}\times \overbrace{(\Theta,\F{S}_\Theta,\P_\Theta)}^{\text{error domain}}
\arrow{d}{\Q}
\arrow{dr}{\YE=\Yq\circ Q}
\arrow{r}{X}
&
\overbrace{(\sX,\sB_\sX)}^{\mathclap{\text{space of Q.o.I.}}}
\\
\underbrace{(\sQ,\sB_\sQ)}_{\mathclap{\text{parameter space of prior}}}
\arrow{r}{\Yq}
& \underbrace{(\sY,\sB_\sY)}_{\mathclap{\text{measurement space}}}
\arrow{l}{\CE{\Q}{\YE}}
\arrow{u}{\CE{X}{\YE}}  &
\end{tikzcd}
\end{equation}

The conditional expectations  $\CE{\Q}{\YE}:\sY\rightarrow\sQ$ of the parameter $Q$ and $\CE{X}{\YE}:\sY\rightarrow\sX$ of a general RV $X$ w.r.t.\ the measurement $\YE$ are discussed in Section~\ref{sec:conditional-expectation_short} and Appendix~\ref{sec:conditional-expectation}.

\begin{problem}
 \label{rem:problem}
 The goal is to calculate the quantities of interest (mean, covariance, etc.) of parameters $\q\in\sQ$ with the information from the measurement, given by the observed value $\cy\in\sY$ and observation error $E$.
 Particularly the vector $\cy$, usually from $\sY=\sR^m$ corresponding to $m$ measurements, is the value taken at the measurement site with the error given by $E$.
\end{problem}
\begin{example}[Identification of loads]
 \label{ex:1}
 As an example, one can consider identification of magnitude of heat source by measuring the response (temperature $u$) of the stationary heat transfer described by
 \begin{align}
 \label{eq:example1}
 \nabla\cdot \nabla u(x) = \sum_{i=1}^n q_if_i(x)\qquad\text{with some boundary conditions (possibly nonlinear)},
 \end{align}
 where $f_i$ are loads and $q_i$ are uncertain quantities for $i=1,\dotsc,n$.
 The information about the system is given by measuring the temperature at measurements points $x_i$ for $i\in\{1,\dotsc,m\}$, which corresponds to an observation operator $\Yq(q) = [u(x_1), u(x_2), \dotsc, u(x_m)]$ for $u$ obtained from model \eqref{eq:example1}.
\end{example}

\subsection{The conditional expectation}
\label{sec:conditional-expectation_short}

In a probabilistic setting, Problem~\ref{rem:problem} to estimate parameters from the measurement is based on Bayes's theorem. Mathematically, it is formulated for two events $A$ and $B$ as subsets of $\Omega$
\begin{align*}
\P(A|B) = \frac{\P(A\cap B)}{\P(B)}= \frac{\P(B|A)}{\P(B)}\P(A),
\end{align*}
where $\P(A)$ is the probability of event $A$, $\P(A\cap B)$ is a probability of simultaneous occurrence of events $A$ and $B$, and $\P(A|B)$ is a conditional probability, i.e.\ a likelihood of event $A$ given that event $B$ has occurred.
However, this is well-defined only when the probability of event $B$ is positive ($\P(B)>0$).

One has to overcome this limitation as well as the fact that Bayes's theorem estimates the posterior probability only w.r.t.\ a single observed event. However one needs to incorporate the new information of the whole set of events.
The proper approach is the conditional expectation w.r.t.\ to a random variable, which is covered in many monographs, e.g.\ \cite{Rao2006book,Durrett2010book}, and also in Appendix~\ref{sec:conditional-expectation}.

Here conditional expectation is summarised for a prior random variable $\Q\in \sL^1(\Omega,\F{S},\P;\sQ)$ with respect to a given measurement $\YE\in\sL^1(\Omega,\F{S},\P;\sY)$. 
However, there is a natural question what one can observe from measurement $\YE$ and how it can be mathematically formulated.
Assuming an event $B$ in $\sY$ has been observed, then the corresponding probability of that is expressed with the help of the push-forward measure
\begin{align*}
\P_{\YE}(B) = \P(\YE^{-1}(B))
\end{align*}
where $\YE^{-1}(B)=\{\omega\in\Omega\sep \YE(\omega)\in B\}$ is a preimage of the set $B$.
In other words, all events in parameter space $\Omega$ that lead to an event $B$ lie in the preimage $\YE^{-1}(B)$.
Therefore, the information from the measurement $\YE$ that can be obtained about the parameters is characterised by all preimages in $\Omega$ for all possible sets $B$, i.e.\ the $\sigma$-algebra generated by the measurement $\YE$ denoted as $\sigma(Y)$, see \eqref{eq:sigmaY} for a definition.

The conditional expectation of any random variable $X:\Omega\rightarrow\sX$, which could be a function of $Q$, w.r.t.\ the measurement $\YE\in\sL^1(\Omega,\F{S},\P;\sY)$ is then defined as a random variable $\CE{X}{\YE}:\sY\rightarrow\sX$ such that it satisfies the orthogonality condition
\begin{align}
\label{eq:CE_duality}
\dual{X-\CE{X}{\YE}\circ \YE}{W\circ \YE}_{\sL^1\times \sL^\infty} &= 0
\quad\forall W\in \sL^\infty(\sY,\sB_\sY,\P_Z;\sX).
\end{align}
The conditional expectation exists for all parameters $X$ in $\sL^1(\Omega,\F{S},\P;\sX)$, which can be proven by the Radon-Nikod\'{y}m theorem \cite{Rao2006book}, and is also unique as an element in $\sL^1(\sY,\sB,\P_Y;\sX)$ (i.e.\ it is unique up to the set of $\P_Y$-measure zero).
When the parameters $X$ have a finite variance, i.e.\ $X\in\sL^2(\Omega,\F{S},\P;\sX)$, the CE minimises the mean square error
\begin{align}
\label{eq:CE_min}
\CE{X}{\YE} = \argmin_{W\in\sL^2(\sY,\sB,\P_Y;\sX)} \|X-W\circ \YE\|^2_{\sL^2(\Omega,\F{S},\P;\sX)},
\end{align}
the variational equation of which is \eqref{eq:CE_duality} in the $\sL^2\times \sL^2$ duality.

\subsection{Bayesian updating using probability densities}\label{sec:bayes-theorem}
Bayesian updating using probability densities is presented here, for details see Appendix~\ref{sec:conditional-probability-and-distribution}.
Particularly taking an open set $A$ in the parameter space $\sQ$, the conditional probability distribution $\P_\Q[A|\YE]$ is defined as the conditional expectation w.r.t.\ a characteristic function of the preimage of $A$, i.e.
\begin{align*}
\P_\Q[A|\YE](\cy) = \CE{\ch{A}\circ Q}{\YE}(\cy).
\end{align*}
As it is a probability measure on the space $\sQ$, it can (in some cases) be expressed with a conditional probability density function $f_{\Q|\YE}:\sQ\rightarrow\sR$ as
$\P_\Q(A|\YE)(\cy) = \int_A f_{\Q|\YE}(q|\cy)\D{q}$.

When the conditional probabilities are expressed in terms of probability densities we obtain a variant of Bayesian updating
\begin{align*}
f_{\Q|\YE}(\q|\cy) = \frac{f_{\YE|\Q}(\cy|\q)}{Z_s}f_\Q(\q)\quad\text{for }\q\in\sQ
\end{align*}
where $f_\Q(\q)$ is the probability density of prior random variable $\Q$, $f_{\YE|\Q}(\cy|\q)$ is a likelihood, and the evidence $Z_s = \int_{\sQ} f_{\YE|\Q}(\cy|\q) f_\Q(\q) \D{\q}$
ensures that the whole conditional probability equals one, i.e.\ that $\int_\sQ f_{\Q|\YE}(\q|\cy)\D{\q}=1$. For the observation with additive error \eqref{eq:measurement}, the likelihood can be expressed in a special form
\begin{align}
\label{eq:likelihood}
f_{\YE|\Q}(\cy|\q) = f_{E}(\cy-\Yq(\q)),
\end{align}
for details see Lemma~\ref{lem:likelihood}.
Altogether we receive the Bayesian updating in the final form
\begin{align}
\label{eq:pdf_update}
f_{\Q|\YE}(\q|\cy) = \frac{f_{E}(\cy-\Yq(\q))}{Z_s}f_\Q(\q).
\end{align}
From PDFs one can calculate the mean and variance of the posterior (conditional) distribution
\begin{subequations}
 \label{eq:pdf_post}
 \begin{align}
 \label{eq:pdf_mean}
 \textrm{Mean}_{\Q|\YE}(\cy) &= \int_{\sQ}\q f_{\Q|\YE}(\q|\cy) \D{\q},
 \\
 \label{eq:pdf_cov}
 \textrm{Covar}_{\Q|\YE}(\cy) &= \int_{\sQ} \bigl[\q-\textrm{Mean}_{\Q|\YE}(\cy)\bigr]\otimes\bigl[\q-\textrm{Mean}_{\Q|\YE}(\cy)\bigr] f_{\Q|\YE}(\q|\cy) \D{\q}
 \end{align}
\end{subequations}
For a future comparison to the conditioned expectation to be developed, we stress here that the conditional mean and covariance depend on the observation value $\cy$.

\subsection{Filtering}\label{sec:filtering}
The whole aim of updating (filtering) the parameters is to obtain a posterior RV with updated (or assimilated) distribution w.r.t the measurement $\YE$ and observed value $\cy\in\sY$. Here several existing filters are discussed.

Certainly, the filter has to contain the conditional expectation $\CE{\Q}{\YE}(\cy)$, which predicts the posterior mean for observations $\cy$, and a zero-mean random variable $R$ determining the posterior distribution.
Naturally, the filter should be a function of the prior but also of the measurement; however, its effective determination remains an open problem.
One approach is based on the orthogonal decomposition of the prior RV $Q = \CE{Q}{\YE}\circ \YE + (Q-\CE{\Q}{\YE}\circ \YE)$ with zero-mean part $R:=(Q-\CE{\Q}{\YE}\circ \YE)$, which leads to a so-called CE mean filter
\begin{align}
\label{eq:filter_mean}
\Q_\text{a}(\xi,\theta)&=\CE{\Q}{\YE}(\cy) + \Q(\xi)-\CE{\Q}{\YE} \circ  \YE(\xi,\theta).
\end{align}
Although this is exact for the specific setting of normally distributed prior and linear observation $\YE$, it fails to predict correct distribution (except mean) for general problems for all possible observations $\cy$.
Particularly the covariance $C_{Q_aQ_a} = C_{RR} = \E{R\otimes R}$ and also higher moments of posterior RV $Q_a$ are determined by the zero-mean part $R$, which in \eqref{eq:filter_mean} is independent of the observation value $\cy$, see also Section~\ref{sec:numerical-example-of-existing-filters} with a numerical example.

Nevertheless the correct covariance can be obtained, similarly as the mean, from the conditional expectation
$C=\CE{\bar{Q}\otimes\bar{Q}}{\YE}(\cy)$ but with respect to the RV $\bar{Q}\otimes \bar{Q}$ for $\bar{Q} = Q-\CE{\Q}{\YE}(\cy)$.
It allows to scale the CE mean filter \eqref{eq:filter_mean} to obtain a filter which can have correct mean and covariance
\begin{align}
\label{eq:filter_cov}
Q_{\ass,\cov}(\xi,\theta) = \CE{\Q}{\YE}(\cy) + C^{1/2}C^{-1/2}_{RR}[\Q(\xi)-\CE{\Q}{\YE} \circ  \YE(\xi,\theta)],
\end{align}
where the inverse is meant in the sense of pseudo-inverse.
Still the filter cannot have correct distribution for general problems.
Actually any filter expressed as a mean plus zero-mean random variable $R=f(Q,\YE)$ depending on the prior and the measurement provides a filter with correct mean and covariance.
The novel filters and their approximations will be studied in following publications.

\subsubsection{Approximation of conditional expectation}\label{sec:approximation-of-conditional-expectation}
Although the filters \eqref{eq:filter_mean} and \eqref{eq:filter_cov} produce posterior RVs $\Q_\ass$ resp. $\Q_{\ass,\cov}$ with only correct conditional mean or conditional mean and covariance, they are not easy to realise numerically as the conditional expectation $\CE{\Q}{\YE}:\sY\rightarrow\sQ$ still has to be approximated.

A simple approximation is an affine approximation of the CE as
\begin{align*}
\CE{\Q}{\YE}(y)\approx \CE{\Q}{\YE}^1(y)=Ky+b
\text{ with }K\in\sR^{\dim\sQ\times\dim\sY}
\end{align*}
leading to a Gauss-Markov-Kalman filter \cite{Matthies2015,Matthies2016}
\begin{align}
\label{eq:filter_Kalman}
\Q_{\ass,1}(\xi,\theta) &= \Q(\xi)-K[\YE(\xi,\theta)]+K[\cy].
\end{align}
The optimal matrix $K$ coincides with the Kalman gain, explicitly expressed as $K=C_{\Q \YE}C_{\YE\YE}^{-1}$ using the covariance matrices $C_{\Q \YE}$ and $C_{\YE\YE}$ of the RVs $\Q$ and $\YE$;
the inverse is considered to be a Moore-Penrose pseudo-inverse.
This filter \eqref{eq:filter_Kalman} has several sampling variants such as the Ensemble Kalman filter \cite{Ernst2015}.
On the other hand, when the RVs $\YE$, $\Q$, and $\Q_{\ass,1}$ are expressed in terms of a polynomial chaos expansion (PCE) --- i.e.\ as a polynomial approximation of a basic random variable --- the PCE variant of the filter is obtained  \cite{Blanchard2010,Pajonk2012,Rosic2013,Rosic2012}.

A more complex approach is based on approximation of the CE with multivariate polynomials $\CE{\Q}{\YE}\approx\CE{\Q}{\YE}^k$ of order $k$, obtained by the minimisation of \eqref{eq:CE_min} over a corresponding space. It leads to an approximation of the CE mean filter \eqref{eq:filter_mean}
\begin{align}
\label{eq:filter_pol}
\Q_{\ass,k}(\xi,\theta)&=\Q(\xi)-\CE{\Q}{\YE}^k[\YE(\xi,\theta)]+\CE{\Q}{\YE}^k[\cy],
\end{align}
for details see \cite{Matthies2015,Matthies2016}.

\subsubsection{Numerical example}\label{sec:numerical-example-of-existing-filters}
\begin{figure}[ht]
 \centering
 \includegraphics[width=\figscalepar\linewidth]{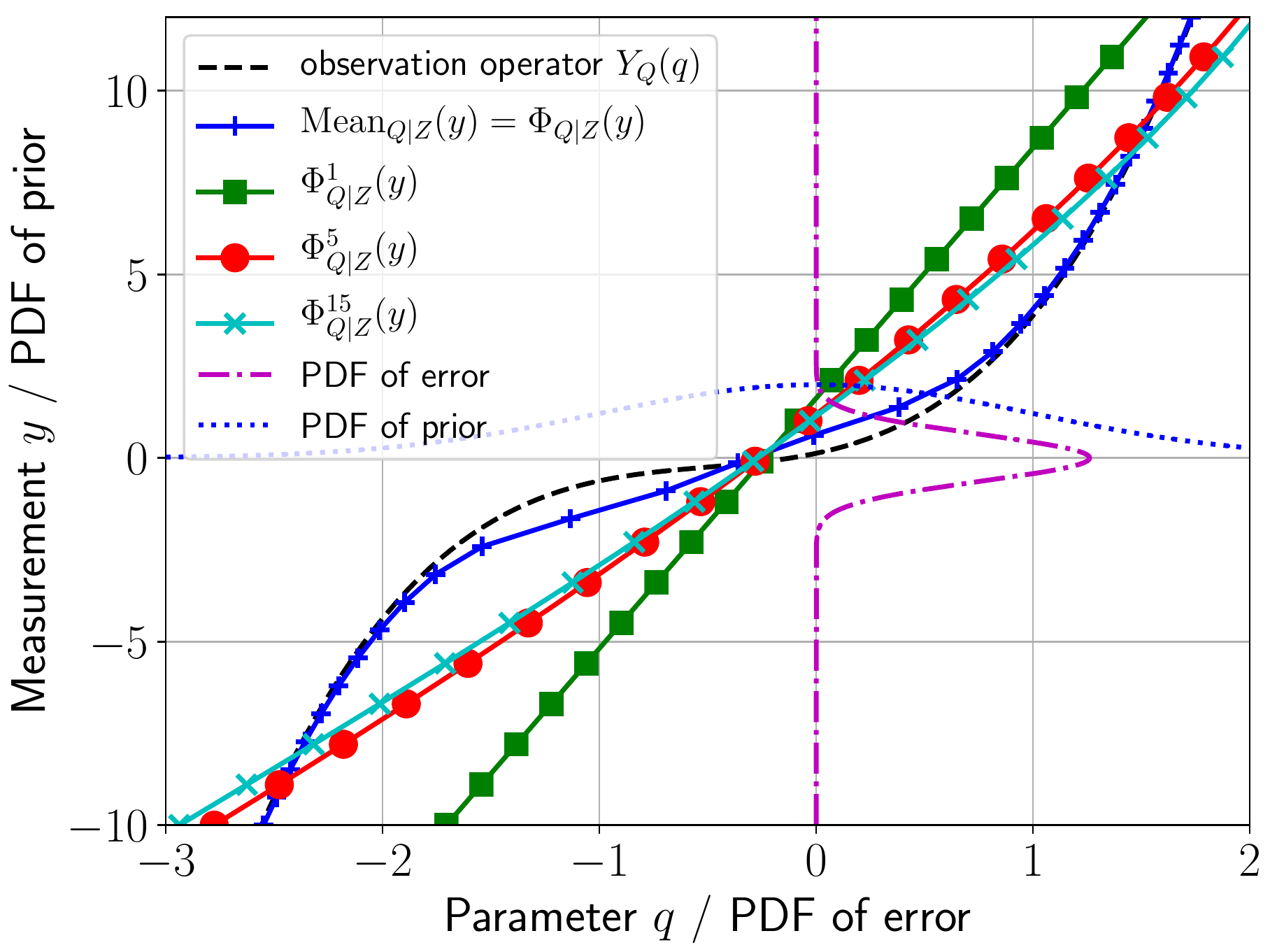}
 \caption{Example of conditional expectation predicting the posterior mean for a one-dimensional example with the spaces $\Xi=\Theta=\sQ=\sY=\sR$, the observation $\Yq(\q) = (\q+\tfrac{1}{2})^3+\tfrac{\q}{2}$, and the probability measures $\P_\Xi$ and $\P_\Theta$ with standard normal distribution $\distN(0,1)$.
  The random variables $\Q\sim\distN(0,1)$ with $\Q(\xi)=\xi$ and $E\sim\distN(0,0.4)$ with $E(\theta)=0.4^{-1/2}\theta$ determine the normal distribution of the parameter $\q=\Q(\xi)\in\sQ$ and of the error.
  The correct posterior mean $\textrm{Mean}_{\Q|\YE}(y)=\CE{\Q}{\YE}(y)$ is calculated from the PDF by \eqref{eq:pdf_mean}, approximations to CE $\CE{\Q}{\YE}^k$ by the minimisation problem \eqref{eq:CE_min} using polynomial approximation of order $k$.
 } 
 \label{fig:mmse_pol}
\end{figure}

One can observe the behaviour of the approximations to the conditional expectation from a simple one-dimensional example that is visualised in Figure~\ref{fig:mmse_pol} and described in detail in the caption of the figure.
The conditional expectation
$\textrm{Mean}_{\Q|\YE}(y)=\CE{\Q}{\YE}(y)$, i.e.\ the map that predicts the posterior mean given observation $y$ has been calculated from the PDF using \eqref{eq:pdf_mean} as accurately as possible,
so that Figure~\ref{fig:mmse_pol} shows the errors in predicted conditional mean due to the various approximations.
This is the difference between $\CE{\Q}{\YE}$, its affine approximation $\CE{\Q}{\YE}^1$, and the higher order approximations $\CE{\Q}{\YE}^5$ and $\CE{\Q}{\YE}^{15}$.

Not surprisingly, the linear approximation $\CE{\Q}{\YE}^1$, which corresponds to the Gauss-Markov-Kalman filter, does not predict the mean correctly.
Also, the higher polynomial approximations $\CE{\Q}{\YE}^k$ with $k=5, 15$ are of no big advantage.
In particular, if the observed value were $\cy=-5$, the linear approximation $\CE{\Q}{\YE}^1$ would provide a big overestimation of the mean, while if the observed value were $\cy=5$, it would provide an underestimation.
Similar remarks apply to the higher order approximations $\CE{\Q}{\YE}^5$ and $\CE{\Q}{\YE}^{15}$.

Analogically, although not depicted here, polynomial approximations of the posterior covariance $\CE{\bar{\Q}\otimes \bar{\Q}}{\YE}^k(y)$ for $\bar{\Q}=Q-\CE{\Q}{\YE}^k(y)$ may be very poor as they also depend on the approximation of the posterior mean.
Additionally, a polynomial approximation with an odd order can always result in a posterior covariance that fails to be  positive semidefinite for certain measurements due to the odd highest order term.

To conclude, the presented filters \eqref{eq:filter_Kalman} and \eqref{eq:filter_pol} may suffer from the poor approximation to the conditional expectation.
Therefore, they may fail to predict even the posterior mean correctly.
This will be overcome in the next section using conditioned expectation $\CE{\Q}{\YE}(\cy)$, i.e. the direct evaluation of the conditional expectation at the value of the measurement $\cy$.

\section{Conditioned expectation and its approximation}\label{sec:simple-conditional-expectation-of-random-fields}

In accordance with the discussion in the previous section, the conditional expectation of $\Q$, while observing $\YE$, predicts only the mean correctly, i.e.\ the value $\CE{\Q}{\YE}(\cy)$ that is called conditioned expectation (CdE).
It is thus superfluous to calculate the whole optimal map $\CE{\Q}{\YE}:\sY\rightarrow\sQ$ of the conditional expectation (CE) when only its value at $\cy$ is needed.
The following theorem shows how to reformulate CE into CdE for a general random variable $X:\Omega\rightarrow\sX$, possibly a function of the parameter $Q$.

\begin{theorem}[Conditioned expectation]
 \label{lem:main_result}
 Let $X\in\sL^1(\Omega,\F{S},\P;\sX)$ be a random variable such that $X(\xi,\cdot)$ is continuous and $\sX$ a finite dimensional space.
 Let $\YE\in\sL^1(\Omega,\F{S},\P;\sY)$ be a measurement \eqref{eq:measurement} with an error $E$, which is bijective (one-to-one map) and has a continuous probability density function $f_E$.
 Then the conditioned expectation  $\CE{X}{\YE}(\cy)$ at the point of observation $\cy$ is equal to
 \begin{align}
 \label{eq:CE_simple}
 \CE{X}{\YE}(\cy) &= \frac{ \int_{\Xi} X[\xi,E^{-1}(\cy-Y(\xi))] f_{E}[\cy-Y(\xi)] \P_\Xi(\D{\xi}) }{ \int_{\Xi} f_E[\cy-Y(\xi)] \P_\Xi(\D{\xi}) }.
 \end{align}
\end{theorem}
Let us make some remarks before the proof will be presented.
\begin{remark}
 Since the value $f_E(\cy-Y(\xi))$ is the likelihood, the conditional expectation equals the weighted average of the random variable with weights being the likelihood \eqref{eq:likelihood}.
\end{remark}
\begin{proof}[Proof of Theorem~\ref{lem:main_result}]
 The proof is based on the approximation of a conditional expectation using simple functions along with a limiting process.
 Since one does not need the whole conditional expectation but only its value at $\cy$, the approximation of CE with a measurable map $\CEsign_{\del}:\sY\rightarrow\sX$ using simple functions is convenient because of their localisation effect.
 Moreover, simple functions are dense on $L^1$ as well as $L^\infty$.
 Starting with $Q=\sR$ for simplicity, the approximation
 \begin{align*}
 \CEsign_\del(y)=\sum_{i=1}^n\alp_i \ch{A_i}(y)
 \end{align*}
 with disjoint sets $A_i$ covering $\sY$ is considered.
 Indeed, there exists an index $j$ such that the observation $\cy$ is contained in $A_j$.
 For notational simplicity we will denote this space $A_j$ as $\Ay=\cy+\Adel$, where $\Adel$ is some $\del$-neighbourhood of $0$; particularly, some neighbourhood from Borel $\sigma$-algebra $\sB_\sY$ is considered, e.g.\ a subset $\prod_{i=1}^M(-\del/2,\del/2)$ centred at $\cy$ in $\sY$.
 
 The approximation of CE with simple functions at the observation $\cy$
 \begin{align*}
 \CEsign_\del(\cy) = \sum_{i=1}^n\alp_i \ch{A_i}(\cy) = \alp_j\ch{\Ay}(\cy)=\alp_j,
 \end{align*}
 equals the particular coefficient $\alp_i$ that can be obtained from the Galerkin approximation of the orthogonality condition \eqref{eq:orthogonality_dual}, i.e.
 \begin{align}
 \label{eq:orghogonality_alp}
 \alp_j \dual{\ch{A_j}\circ \YE}{\ch{A_i}\circ \YE}_{\sL^1\times \sL^\infty} &=
 \dual{X}{\ch{A_i}\circ \YE}_{\sL^1\times \sL^\infty}\quad\forall i=1,\dotsc,n.
 \end{align}
 This is a fully decoupled (diagonal) linear system for $\alp_i$ when it is tested with the characteristic functions.
 
 It can be simplified by integration over the support of the random variable $\ch{\Ay}\circ \YE$, i.e.\ a set
 \begin{align*}
 \Omega_{\delta}
 &=
 \{(\xi,\theta) \sep \xi\in\Xi \text{ and } \theta\in\Theta_\delta(\xi)\}
 \end{align*}
 where $\Theta_\delta(\xi)$ is a preimage of $\Ay$ w.r.t.\ error $E$, namely
 \begin{align*}
 \Theta_\delta(\xi) &= E^{-1}(\Adel + \cy - Y(\xi))=\{\theta\in\Theta \sep E(\theta) \in \Adel + \cy - Y(\xi)\}.
 \end{align*}
 The orthogonality \eqref{eq:orghogonality_alp} leads to a simple approximation of the conditional expectation
 \begin{align*}
 \alp_j=\CEsign_\delta(\cy) 
 &= \frac{ \int_{\Omega_\del} X(\omega) \P(\D{\omega}) }{ \int_{\Omega_\del}  \P(\D{\omega}) }
 = \frac{ \int_{\Xi} [\int_{\Theta_\delta(\xi)} X(\xi,\theta) \P_\Theta(\D{\theta})]  \P_\Xi(\D{\xi}) }{ \int_\Xi [\int_{\Theta_\delta(\xi)} \P_\Theta(\D{\theta})] \P_\Xi(\D{\xi})}.
 \end{align*}
 The conditional expectation can be then calculated as a limit for $\delta\rightarrow 0$, thanks to the density of simple functions.
 The integrals over $\Theta_\delta(\xi)$ can be averaged with the volume of $\Adel$ (i.e.\ with Lebesgue measure $\lambda$ of $\Adel$) in both nominator and denominator. Then the limit in the nominator is calculated as
 \begin{align*}
 \frac{1}{\meas{\Adel}} &\int_{\Theta_\del(\xi)} X(\xi,\theta)\P_\Theta(\D{\theta})
 \stackrel{\eqref{eq:change_of_var}}{=} \frac{1}{\meas{\Adel}} \int_{\Adel+\cy-Y(\xi)} X(\xi,E^{-1}(y)) \P_E(\D{y})
 \\
 &=
 \frac{1}{\meas{\Adel}} \int_{\Adel+\cy-Y(\xi)} X(\xi,E^{-1}(y)) f_E(y) \D{y}
 \stackrel{\del\rightarrow 0}{\longrightarrow} X[\xi,E^{-1}(\cy-Y(\xi))] f_E[\cy-Y(\xi)].
 \end{align*}
 The limit in the denominator is calculated as for $X(\xi,\theta)=1$ to obtain a likelihood $f_E[\cy-Y(\xi)]$. Both together yield the required formula \eqref{eq:CE_simple} for $\sX=\sR$.
 
 Now, the general case for finite dimensional $\sX$ will be proven.
 Taking $v\in\sX$ one can apply the proof on the scalar-valued random variable $X\cdot v$. The linearity of conditional expectation and of the integral leads to the required \emph{conditioned expectation}
 \begin{align*}
 v\cdot\CE{X}{\YE}(\cy) &= v\cdot\frac{ \int_{\Xi} X(\xi) f_{E}[\cy-Y(\xi)] \P_\Xi(\D{\xi}) }{ \int_{\Xi} f_E[\cy-Y(\xi)] \P_\Xi(\D{\xi}) }.
 \end{align*}
 The proof is finished since it holds for all $v\in\sX$.
\end{proof}
\begin{remark}
 We note that the space $\sL^1(\Omega,\F{S},\P;\sQ)$ of parameters is defined as a product space $\sL^1(\Xi\times\Theta,\F{S}_\Xi\times\F{S}_\Theta,\P_\Xi\times\P_\Theta;\sQ)$, where $\Xi$ is the original domain of the parameter $\Q$ while $\Theta$ is the domain for a random variable of error $E$.
 This means that the error is assumed to be independent of the parameters.
 The conditioned expectation presented here \eqref{eq:CE_simple} is valid for a special also case when $X(\xi,\theta)=\Q(\xi)$, see Corollary~\ref{lem:main_res_sim}, as well as for the more general random variable $X$, depending also on the error variable $\theta$.
 This could be useful when estimating the whole distribution of the updated random variable and not only the mean.
\end{remark}

\begin{remark}[Requirements on error E] The requirement on $E$ to be one-to-one can be fully omitted when the random variable $X$ depends only on the variable $\xi\in\Xi$, which is stated in the following simplified corollary.
 Nevertheless, the assumption is also satisfied in practical situations.
 For example when an error $E:\Theta\rightarrow\sY$ for $\Theta=\sY=\sR^m$ is assumed to have $m$ independent measurements, i.e.\ $E(\theta)=[E_1(\theta_1),E_2(\theta_2),\dotsc,E_m(\theta_m)]$, and when the individual errors $E_i(\theta_i)$ with $i\in\{1,2,\dotsc,\}$ have a continuous distribution $f_{E_i}$, typically assumed Gaussian or close to Gaussian.
 The distribution is thus expressed as the product of the individual ones $f_E(y)=\prod_{i=1}^m f_{E_i}(y_i)$.
 
 The probability densities exist
 \begin{align*}
 f_{E_i}(\q) = \sum_{i=1}^{n}\frac{f_{\theta_i}(r_i)}{|X'(r_i)|}
 \end{align*}
 also when the individual error terms are expressed in terms of polynomial chaos expansion, i.e.\ as a polynomial $E_i(\theta_i)=\sum_{j=0}^k \alp_j \theta_i^j$,
 where $r_i$ are roots of the polynomial $E_i(\theta)-\q=0$ and $f_{\theta_i}$ is a probability density of variable $\theta_i$ (normal distribution in case of Hermite polynomials).
\end{remark}

\begin{corollary}
 \label{lem:main_res_sim}
 Let $X\in\sL^1(\Omega,\F{S},\P;\sQ)$ be a random variable such that it depends only on the subdomain $\Xi$ of $\Omega=\Xi\times\Theta$, i.e.\ $X(\xi,\theta)=X(\xi)$. Let $\YE=(Y+E)\in\sL^1(\Omega,\F{S},\P;\sY)$ be a measurement with an error $E$, which has a continuous probability density function $f_E$. Then the conditioned expectation  $\CE{X}{\YE}(\cy)$ at the point of observation $\cy$ equals
 \begin{align*}
 \CE{X}{\YE}(\cy) &= \frac{ \int_{\Xi} X(\xi) f_{E}[\cy-Y(\xi)] \P_\Xi(\D{\xi}) }{ \int_{\Xi} f_E[\cy-Y(\xi)] \P_\Xi(\D{\xi}) }.
 \end{align*}
\end{corollary}
\subsection{The connection of conditioned expectation with probability densities}
Here we discuss the connection with conditioned expectation (Theorem~\ref{lem:main_result}) and the evaluation of mean \eqref{eq:pdf_mean} and covariance \eqref{eq:pdf_cov} based on probability densities.
\begin{lemma}
 \label{lem:connect_SCE_pdf}
 The mean \eqref{eq:pdf_mean} and covariance \eqref{eq:pdf_cov}, obtained from Bayesian updating using probability densities, are equal to conditioned expectations
 \begin{subequations}
  \begin{align}
  \label{eq:SCE_mean}
  \CE{\Q}{\YE}(\cy) &= \frac{ \int_{\Xi} \Q(\xi) f_{E}[\cy-Y(\xi)] \P_\Xi(\D{\xi}) }{ \int_{\Xi} f_{E}[\cy-Y(\xi)] \P_\Xi(\D{\xi}) }
  \\
  \label{eq:SCE_cov}
  \CE{\bar{\Q}\otimes \bar{\Q}}{\YE}(\cy) &= \frac{ \int_{\Xi} \bar{\Q}(\xi)\otimes \bar{\Q}(\xi) f_{E}[\cy-Y(\xi)] \P_\Xi(\D{\xi}) }{ \int_{\Xi} f_{E}[\cy-Y(\xi)] \P_\Xi(\D{\xi}) }
  \end{align}
 \end{subequations}
 for $\bar{\Q}=\Q-\CE{\Q}{\YE}(\cy)$, respectively.
\end{lemma}
\begin{proof}
 The calculation follows directly from the change of the variable formula in the integral~\ref{eq:change_of_var}.
 For clarity, the statements is explicitly written for a function $X(\sQ)$ of $\sQ$
 \begin{align*}
 \CE{X\circ\Q}{\YE}(\cy) &= \frac{ \int_{\Xi} X(\Q(\xi)) f_{E}[\cy-Y(\xi)] \P_\Xi(\D{\xi}) }{ \int_{\Xi} f_{E}[\cy-Y(\xi)] \P_\Xi(\D{\xi}) } = \frac{\int_{\sQ} X(\q) f_{E}[\cy-\Yq(\q)]\P_\Q(\D{\q})}{\int_{\sQ} f_{E}[\cy-\Yq(\q)] \P_\Q(\D{\q})}
 \\
 &= \frac{\int_{\sQ} X(\q) f_{E}[\cy-\Yq(\q)]f_\Q(\q)\D{\q}}{\int_{\sQ} f_{E}[\cy-\Yq(\q)] f_\Q(\q) \D{\q}}
 \end{align*}
 where a push-forward measure $\P_\Q$ is defined as $\P_\Q(A)=\P[\Q^{-1}(A)]$ for $A\in\sB_\sQ$, and the probability density $f_\Q$ is defined with respect to the Lebesgue measure. The formula for the posterior mean is obtained with $X(Q) = Q$ and for the posterior covariance with $X(Q) = (Q-\CE{\Q}{\YE}(\cy))\otimes (Q-\CE{\Q}{\YE}(\cy))$.
\end{proof}

\subsection{Numerical approximation of conditioned expectation}
The conditioned expectation, stated in Theorem~\ref{lem:main_result}, has to be integrated over the domain $\Xi$.
In practise, the integrals can be evaluated approximately
\begin{align}
\label{eq:CE_numerical}
\aCE{X}{\YE}(\cy)
\approx 
\frac{\sum_{i}w_{i} X[\xi_i,E^{-1}(\cy-Y(\xi_i))] f_{E}[\cy-Y(\xi_i)]}{\sum_{j}w_{j} f_{E}[\cy-Y(\xi_j)]}
\end{align}
using some numerical integration rule with integration points $\xi_i\in\Xi$ and corresponding integration weights $w_i\in\sR$ for $i\in\{1,2,\dotsc,k\}$, which are expressed with respect to the probability measure $\P_\Xi$.
As an alternative, the Monte Carlo method can be used with randomly chosen integration points, again w.r.t.\ $\P_\Xi$, and equal weight $w_i=\frac{1}{n}$.

The natural question arises about the positive-definiteness of the approximated covariance. The positive answer is summarised in the lemma.
We emphasise that the covariance computed with the linear $\CE{\bar{Q}\otimes\bar{Q}}{\YE}^1(\cy)$ or non-linear $\CE{\bar{Q}\otimes\bar{Q}}{\YE}^k(\cy)$ approximation to the CdE is not guaranteed to be positive semi-definite, and often it fails.
\begin{lemma}[Covariance]
 \label{lem:CE_covar}
 Let $\xi_i$ be the points of a numerical integration with positive weights $w_i>0$ for $i\in\{1,2,\dotsc,N\}$. Then the approximated covariance is positive semi-definite.
\end{lemma}
\begin{proof}
 Let $C\in\sQ\otimes\sQ$ be an approximation to a covariance matrix, i.e.
 \begin{align}
 \label{eq:CE_covar}
 C = \aCE{\bar{\Q}\otimes\bar{\Q}}{\YE}(\cy) 
 &=\frac{\sum_{j}w_{j} \bar{\Q}(\xi_j)\otimes \bar{\Q}(\xi_j) f_{E}[\cy-Y(\xi_j)] }{ \sum_{j}w_{j} f_{E}[\cy-Y(\xi_j)]}
 \end{align}
 where $\bar{\Q}=\Q-\CE{\Q}{\YE}(\cy)$. Then for every $\q\in\sQ$, we can calculate
 \begin{align*}
 \scal{Cq}{q}_\sQ = \frac{\sum_{j}w_{j} \scal{\bar{\Q}(\xi_j)\otimes \bar{\Q}(\xi_j)q}{q}_\sQ f_{E}[\cy-Y(\xi_j)] }{ \sum_{j}w_{j} f_{E}[\cy-Y(\xi_j)]} \geq 0
 \end{align*}
 because $\scal{\bar{\Q}(\xi_j)\otimes \bar{\Q}(\xi_j)q}{q}_\sQ=\scal{\bar{\Q}(\xi_j) q}{\bar{\Q}(\xi_j)q}_\sQ>0$ for all $j$.
\end{proof}

\begin{remark}[Parameters as PCE]
 In many cases the parameters $\Q$ are approximated in terms of polynomial chaos expansion (PCE), i.e.\ $Q(\xi) = \sum_{k=1}^p Q_k \psi_k(\xi)$. In the numerical integration of CdE it requires only the additional evaluation of polynomials
 \begin{align*}
 \aCE{\Q}{\YE}(\cy) &= \frac{ \sum_{j} \sum_{k=1}^p  w_j Q_k \psi_k(\xi_j) f_{E}[\cy-Y(\xi_j)] }{ \sum_j f_{E}[\cy-Y(\xi_j)]  }.
 \end{align*}
\end{remark}
\begin{remark}[Parameters as samples] Another option for parameters is when they are given as (Monte Carlo) samples $\bigl(Q(\xi_i)\bigr)_{i=1}^N$. In that case the samples are used for integral evaluation in \eqref{eq:CE_numerical}.
\end{remark}
\section{Numerical examples}\label{sec:numerical-examples}

\begin{figure}[h]
 \centering
 \includegraphics[width=\figscalepar\linewidth]{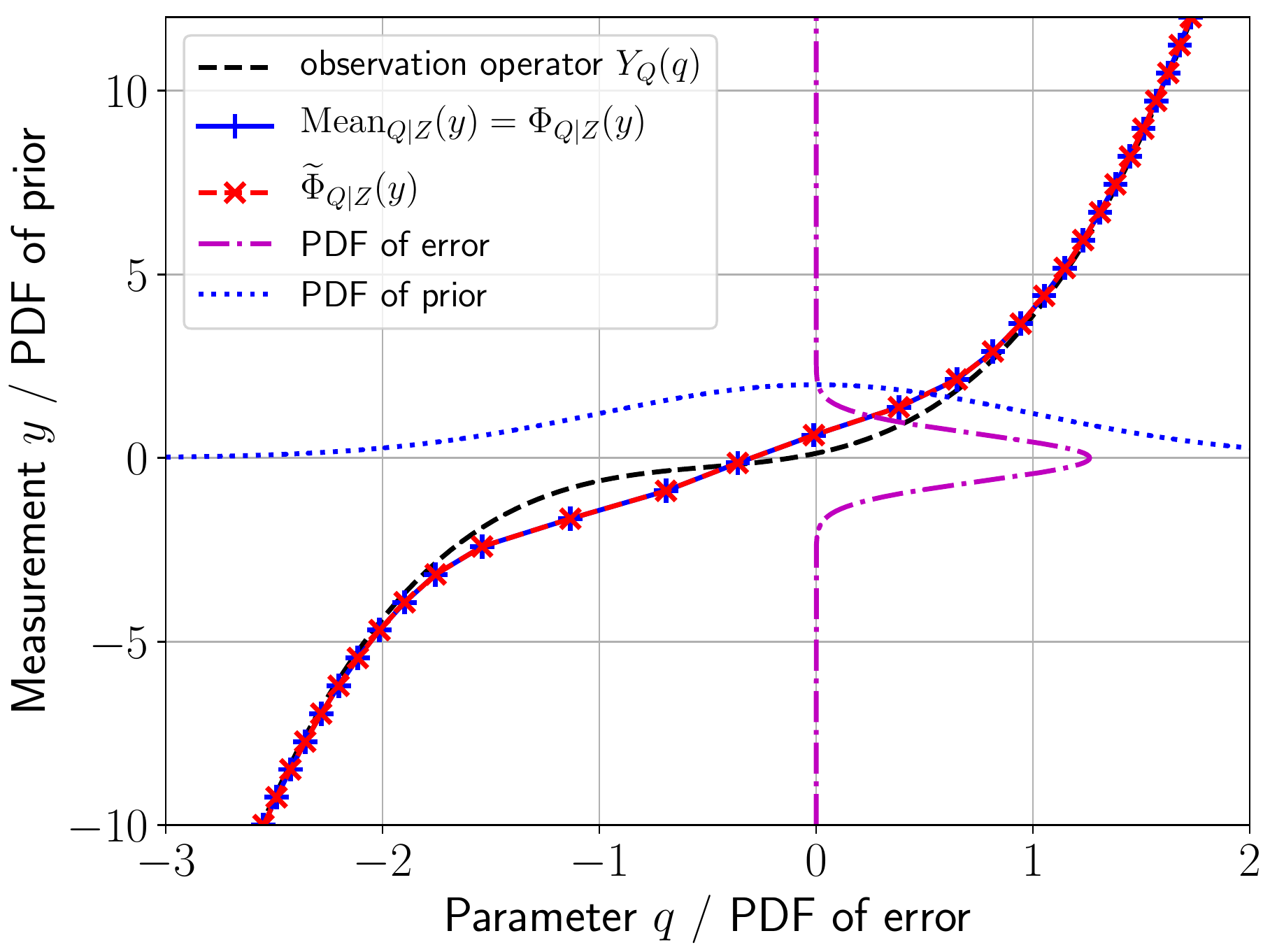}
 \includegraphics[width=\figscalepar\linewidth]{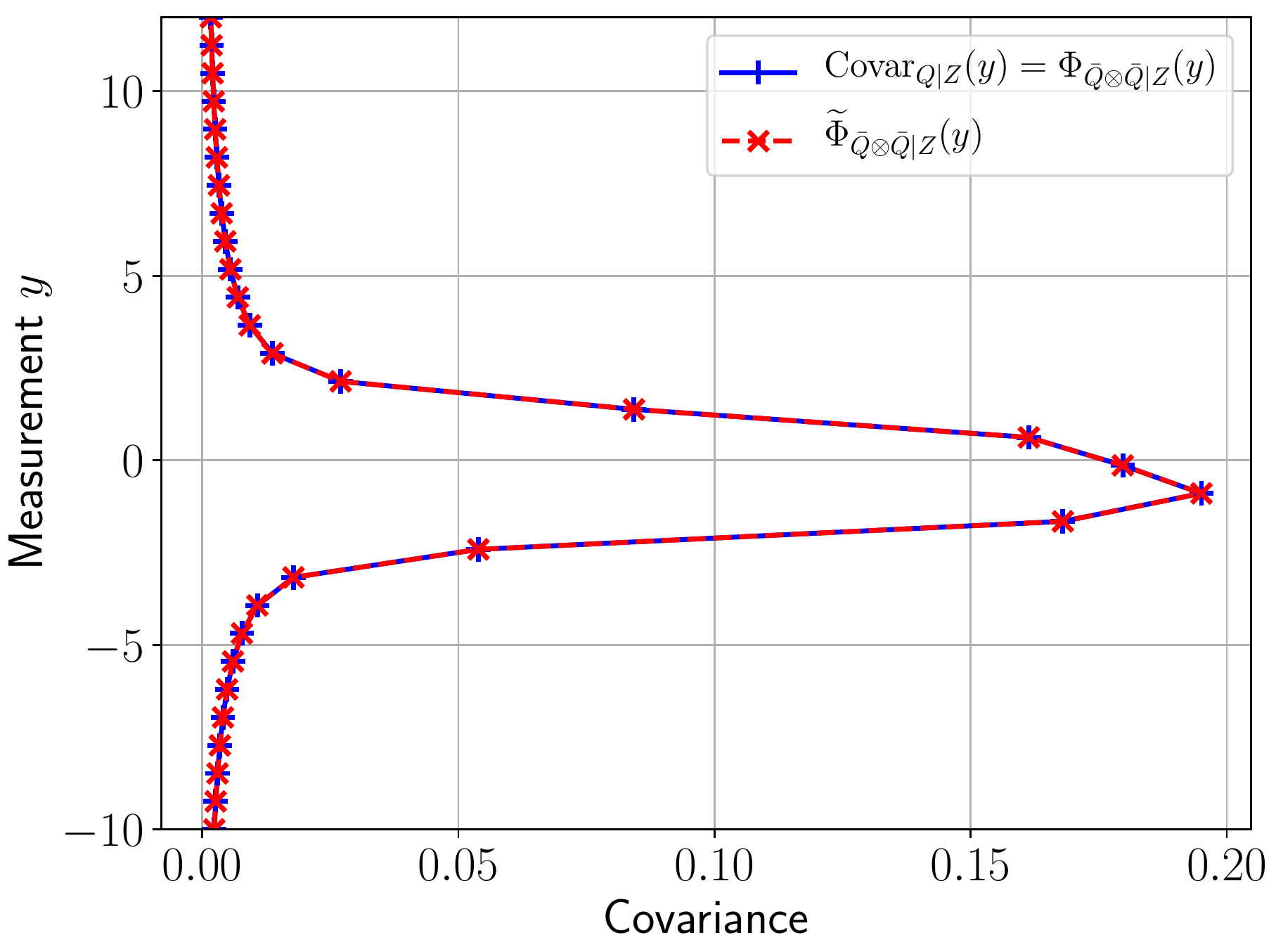}
 \caption{Example of conditional expectation predicting the posterior mean (left) and variance (right) for a one-dimensional example with the spaces $\Xi=\Theta=\sQ=\sY=\sR$, the measurement $\Yq(\q) = (\q+\tfrac{1}{2})^3+\tfrac{\q}{2}$, and the probability measures $\P_\Xi$ and $\P_\Theta$ with standard normal distribution $\distN(0,1)$.
  The random variables $\Q\sim\distN(0,1)$ with $\Q(\xi)=\xi$ and $E\sim\distN(0,0.4)$ with $E(\theta)=0.4^{1/2}\theta$ determine the normal distribution of the parameter $\q=\Q(\xi)\in\sQ$ and of the error.
  The correct posterior mean $\textrm{Mean}_{\Q|\YE}(y)=\CE{\Q}{\YE}(y)$ and covariance $\textrm{Covar}_{\Q|\YE}(y)=\CE{\bar{Q}\otimes\bar{Q}}{\YE}(y)$ are calculated from the PDF by \eqref{eq:pdf_post} and conditioned expectations of mean $\aCE{\Q}{\YE}(y)$ and covariance $\aCE{\bar{\Q}\otimes\bar{\Q}}{\YE}(y)$ with $\bar{\Q}=\Q-\CE{\Q}{\YE}(y)$ by
  \eqref{eq:CE_numerical}.}
 \label{fig:1d_Q3}
\end{figure}

Numerical examples that are presented here to confirm the theoretical findings were conducted using Python implementation StoPy that is freely available at \href{https://github.com/vondrejc/StoPy}{https://github.com/vondrejc/StoPy}.
All the numerical values were calculated as exactly as possible to suppress any error stemming from the numerical integration.
Particularly, we have utilised Gauss-Hermite quadrature of sufficiently high order, the QUADPACK library \cite{Favati1991} for an automated integration, or the Monte-Carlo method to double-check the results.

The first example is exactly the same as in Figure~\ref{fig:mmse_pol} that compares the Gauss-Markov-Kalman filter with higher-order polynomial approximations of conditional expectation.
The results presented in Figure~\ref{fig:1d_Q3} clearly indicate that the conditioned expectation predicts the mean as well as the variance accurately.

\begin{figure}[h]
 \centering
 \includegraphics[width=\figscalepar\linewidth]{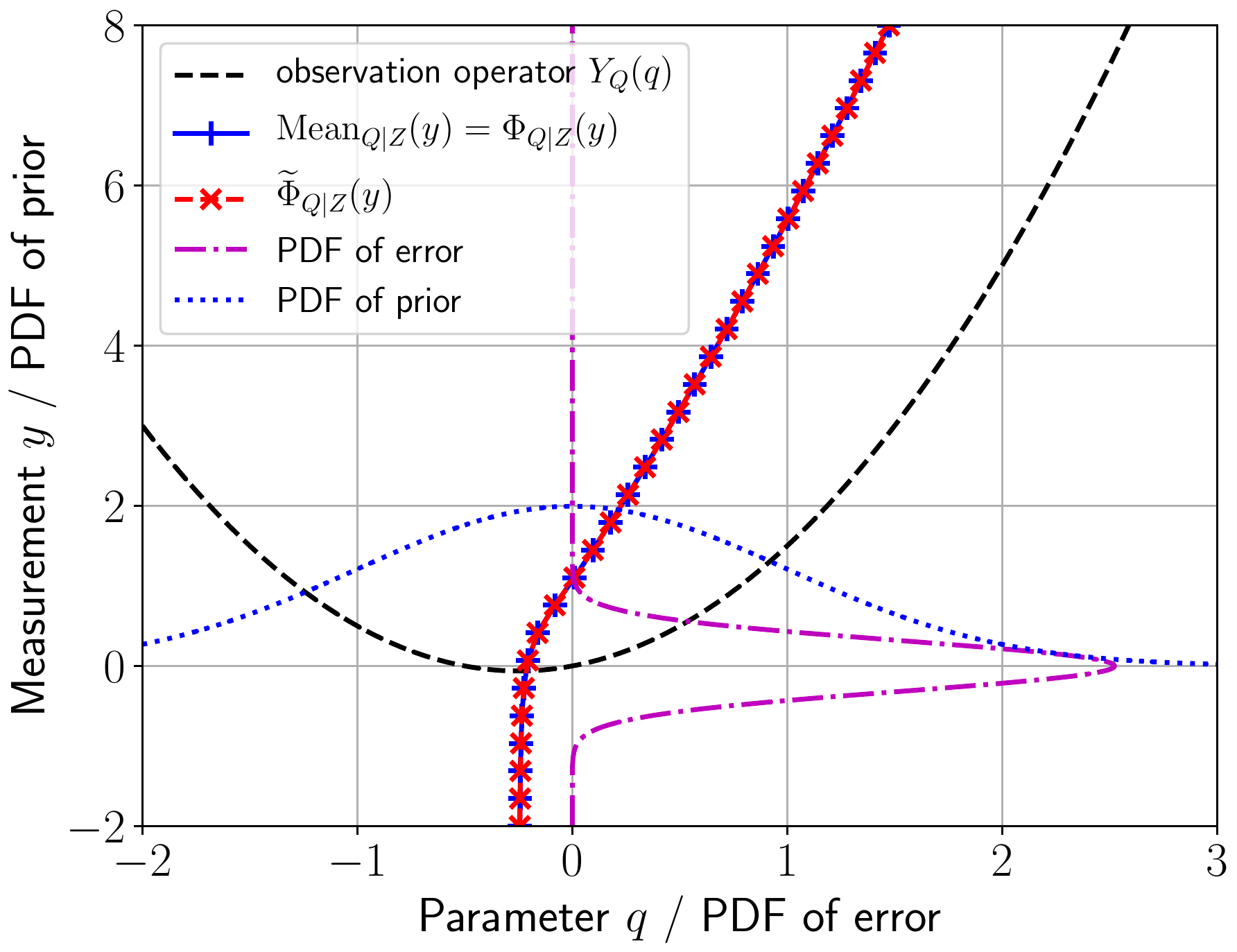}
 \includegraphics[width=\figscalepar\linewidth]{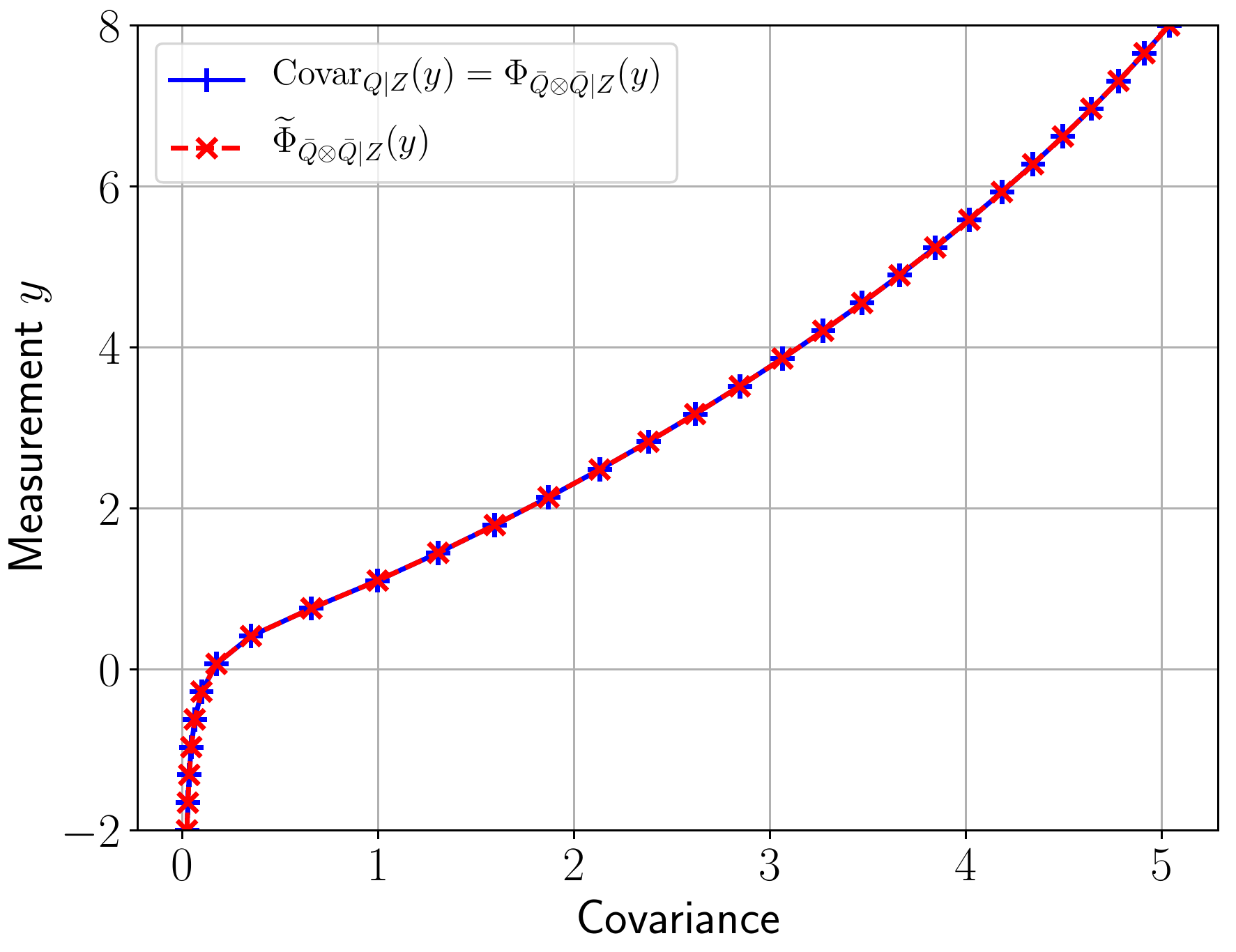}
 \caption{Example of conditional expectation predicting the posterior mean (left) and variance (right) for a simple 1D example with the spaces $\Xi=\Theta=\sQ=\sY=\sR$, the measurement $\Yq(\q) = \q^2+\tfrac{\q}{2}$,  and the probability measures $\P_\Xi$ and $\P_\Theta$ with standard normal distribution $\distN(0,1)$.
  The random variables $\Q\sim\distN(0,1)$ with $\Q(\xi)=\xi$ and $E\sim\distN(0,10^{-1})$ with $E(\theta)=10^{-1/2}\theta$ determine the normal distribution of the parameter $\q=\Q(\xi)\in\sQ$ and of the error.
  The correct posterior mean $\textrm{Mean}_{\Q|\YE}(y)=\CE{\Q}{\YE}(y)$ and covariance $\textrm{Covar}_{\Q|\YE}(y)=\CE{\bar{Q}\otimes\bar{Q}}{\YE}(y)$ are calculated from the PDF by \eqref{eq:pdf_post} and conditioned expectations of mean $\aCE{\Q}{\YE}(y)$ and covariance $\aCE{\bar{\Q}\otimes\bar{\Q}}{\YE}(y)$ with $\bar{\Q}=\Q-\CE{\Q}{\YE}(y)$ by
  \eqref{eq:CE_numerical}.
 }
 \label{fig:1d_Q2}
\end{figure}

The second example, presented in Figure~\ref{fig:1d_Q2}, is an example with a non-monotonous solution operator that leads to double peaks of the conditional distribution.
Again, it clearly confirms that conditional mean and covariance are accurately predicted by conditioned expectation.

\begin{figure}
 \centering
 \includegraphics[width=\figscalepar\linewidth]{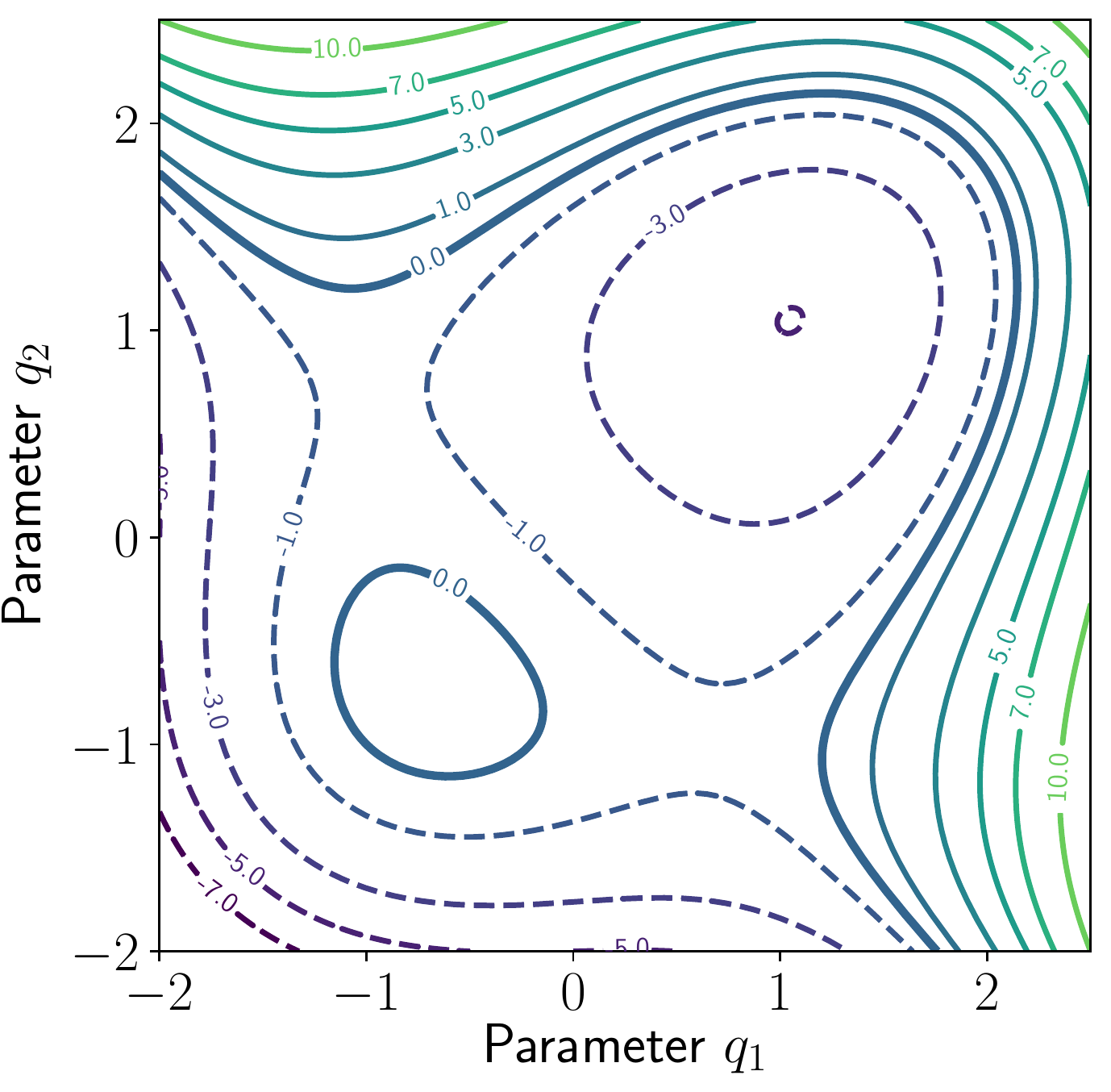}
 \caption{A response surface for a two-dimensional scalar example with spaces $\Xi=2$ and $\Theta=\sQ=\sY=\sR$, the measurement $\Yq(\q) = \q_1(\q_1+\tfrac{3}{2})(\q_1-\tfrac{3}{2})+\q_2(\q_2+\tfrac{3}{2})(\q_2-\tfrac{3}{2})-\q_1\q_2$, the observed value $\cy=0$ emphasised with bold, and the probability measures $\P_\Xi$ and $\P_\Theta$ with standard (multivariate) normal distribution, i.e.\ $\distN([0,0],\mathrm{diag}([1,1]))$ and $\distN(0,1)$, respectively.
  The RVs $\Q\sim\distN([0,0],\mathrm{diag}([1,1]))$ with $\Q(\xi)=[\xi_1,\xi_2]$ and $E\sim\distN(0,0.4)$ with $E(\theta)=0.4^{1/2}\theta$ determine the normal distribution of the parameter $\q=\Q(\xi)\in\sQ$ and of the error.}
 \label{fig:fig2d_response}
\end{figure}

\begin{figure}[h]
 \centering
 \includegraphics[width=\figscalepar\linewidth]{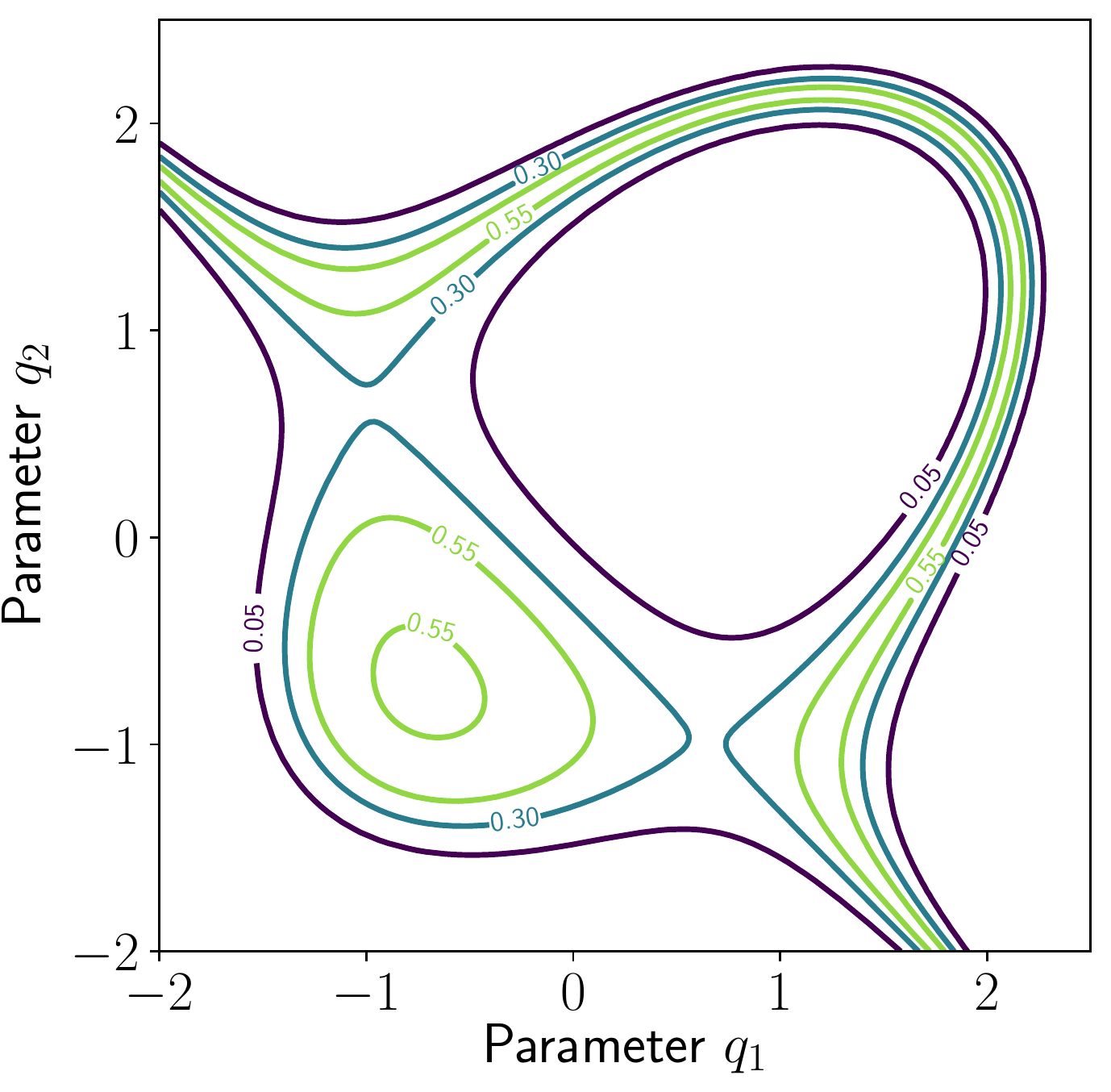}
 \includegraphics[width=\figscalepar\linewidth]{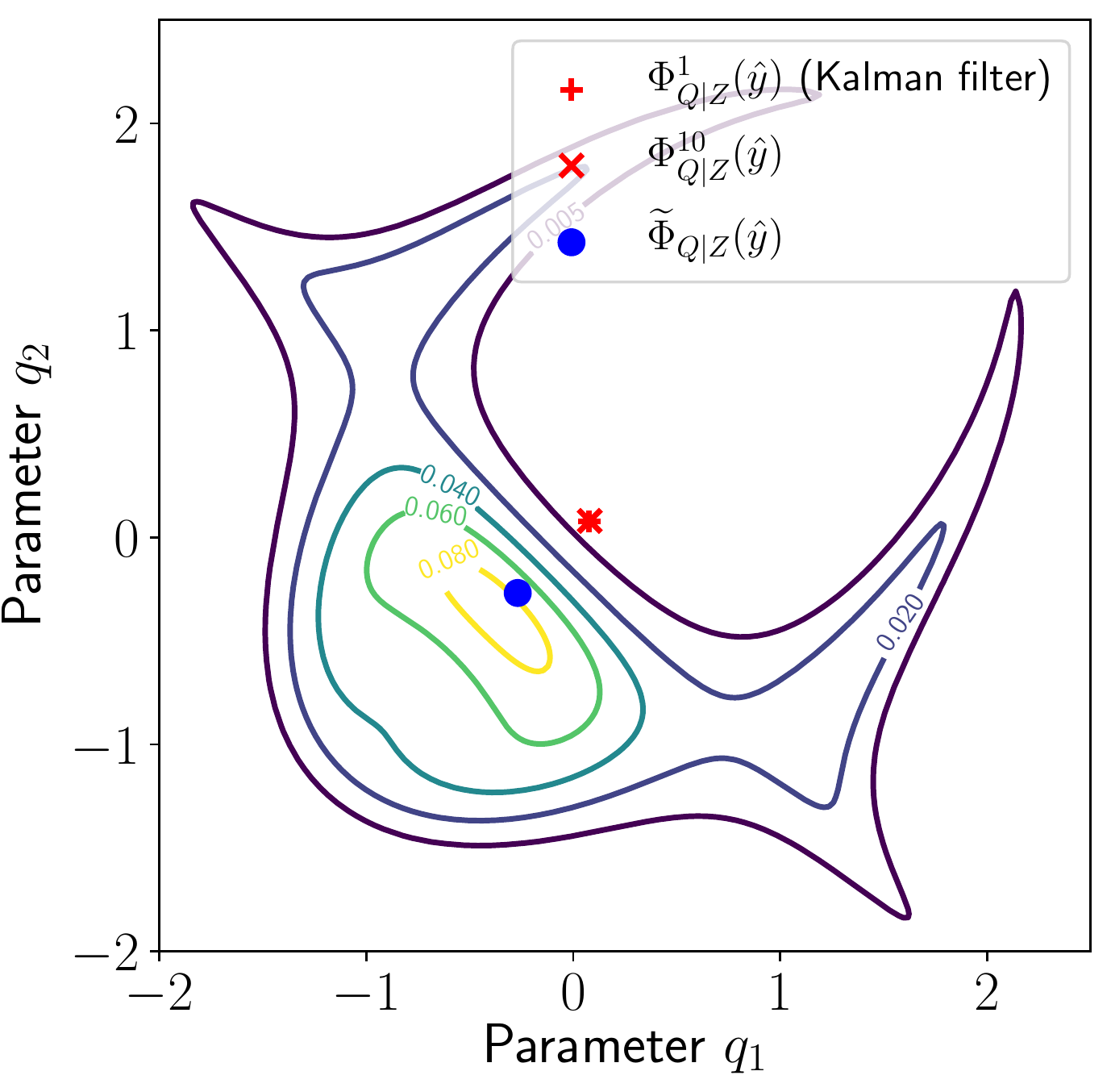}
 \caption{Bayesian updating for the two-parameter problem from Figure~\ref{fig:fig2d_response} with likelihood (left) and the probability density of the posterior distribution \eqref{eq:pdf_update} (right), which also shows the accurate conditioned expectation $\CE{\Q}{\YE}(\cy)$ of the mean and its polynomial approximations \eqref{eq:filter_pol}.
 }
 \label{fig:2d}
\end{figure}

The last example focuses on a problem with two parameters and one measurement, i.e.\ when $\sY=\sR$ and $\q$ belongs to $\sQ=\sR^2$ and.
This example is described in the caption of Figure~\ref{fig:fig2d_response}, which shows a response surface of the forward problem.
The likelihood function \eqref{eq:likelihood} and the posterior distribution (in terms of probability density \eqref{eq:pdf_update}) are depicted in Figure~\ref{fig:2d}.
Particularly, it shows the conditional expectation and its approximations that predict or approximate the updated mean; the concrete values are summarised here
\begin{align*}
\aCE{\Q}{\YE}(\cy) &=-[0.2834, 0.2834],
&
\CE{\Q}{\YE}^1(\cy) &= [ 0.0775,  0.0775],
&
\CE{\Q}{\YE}^5(\cy)&=[ 0.0805,  0.0805],
\\
\CE{\Q}{\YE}^{10}(\cy)&=[ 0.0795,  0.0795],
&
\CE{\Q}{\YE}^{15}(\cy)&=[ 0.0781,  0.0781],
&
\CE{\Q}{\YE}^{20}(\cy)&=[ 0.0775,  0.0775].
\end{align*}
One can notice the big difference between conditioned expectation $\CE{\Q}{\YE}(\cy)$ predicting the mean accurately and its polynomial approximation.
Particularly, the conditioned expectation predicts the covariance matrix
\begin{align*}
\aCE{\bar{\Q}\otimes\bar{\Q}}{\YE}(\cy) = 
\left[\begin{matrix}0.6132200662801 & -0.1438067291666\\-0.1438067291666 & 0.6132200662801\end{matrix}\right]
\end{align*}
which has positive eigenvalues $0.46941334$ and $0.7570268$ corresponding to eigenvectors $[1, 1]$ and $[1,-1]$, i.e.\ the axes of the quadrants in the graph.

\clearpage
\section{Conclusion}\label{sec:conclusion}
This paper is focused on the computation of conditional expectation (CE) which is the proper approach in inverse problems using Bayesian updating in terms of random variables.
It is only necessary to evaluate the conditional expectation $\CE{\Q}{\Z}$ at the point of measurement $\cy$ in order to obtain conditional probability distribution.
Therefore, under the assumption that the observation is accompanied with an additive error, the CE is reformulated to express directly the value $\CE{\Q}{\Z}(\cy)$, here called a \emph{conditioned expectation} (CdE); the main result is stated in Theorem~\ref{lem:main_result} and Corollary~\ref{lem:main_res_sim}.
The consequences and other outcomes are summarised as follows:
\begin{itemize}
 \item CdE predicts correctly a conditional mean \eqref{eq:SCE_mean} and covariance \eqref{eq:SCE_cov}, which overcomes the linear approximation of CE used in the Gauss-Markov-Kalman filter \eqref{eq:filter_Kalman} or CE approximated with multivariate polynomials \eqref{eq:filter_pol}, see numerical example in Figure~\ref{fig:mmse_pol}.
 
 \item Numerical implementation \eqref{eq:CE_numerical} of CdE is straight-forward and requires only a numerical integration over the domain of the prior random variable; the approximation of CE with polynomials requires also integration over the error domain.
 
 \item For a numerical integration scheme with positive weights, the approximated covariance matrix \eqref{eq:CE_covar} using CdE is always positive semi-definite (Lemma~\ref{lem:CE_covar}).
 
 \item The connection between CdE and Bayesian updating in terms of probability measures (densities) is established (Lemma~\ref{lem:connect_SCE_pdf}). It is essentially a change of variables.
\end{itemize}

\appendix
\section{Conditional expectations, probabilities, and distributions}\label{sec:conditional-expectation}

This section contains an abstract description of conditional expectation, conditional probability, and conditional distribution; the diagram of the spaces and maps can be found in \eqref{diag:commutative_diag}.
When it is useful for a better understanding, the proof or its outline is presented.

\subsection{Conditional expectation}
\begin{definition}
 Let $X\in \sL^1(\Omega,\F{S},\P;\sX)$ be a random variable and $Z\in\sL^1(\Omega,\F{S},\P;\sY)$ be a measurement with Borel $\sigma$-algebra $\sB$ on $\sY$. Let 
 \begin{align}
 \label{eq:sigmaY}
 \sigma(Z)=Z^{-1}(\sB) = \{Z^{-1}(B)\sep B\in\sB\}
 \end{align}
 be the $\sigma$-algebra generated by $Z$. 
 Then the conditional expectation of $X$ w.r.t.\ $Z$ is a random vector $\CE{X}{Z}$ from $\sL^1(\sY,\sB,\P_Z;\sX)$ satisfying
 \begin{align}
 \label{eq:CE_Y}
 \int_{A} X(\omega)\P(\D{\omega}) = \int_{A} \CE{X}{Z}\circ Z(\omega)\P(\D{\omega})
 \quad\forall A\in\sigma(Z).
 \end{align}
\end{definition}
\begin{remark}
 \label{rem:EsigY}
 In many manuscripts, the conditional expectation $\E{X|\sigma(Z)}:\Omega\rightarrow\sX$ w.r.t.\ the $\sigma$-algebra $\sigma(Z)$ satisfying
 $\int_A X(\omega)\P(\D{\omega}) = \int_A \E{X|\sigma(Z)}(\omega)\P(\D{\omega})$ for all $A\in\sigma(Z)$
 is introduced instead of $\CE{X}{Z}$. However, they are simply related as $\E{X|\sigma(Z)} = \CE{X}{Z}\circ Z$.
\end{remark}
\begin{lemma}
 The conditional expectation from the previous definition exists and is unique up to $\P_Z$-almost everywhere equality.
\end{lemma}
\begin{proof}
 The existence of conditional expectation $\CE{X}{Z}$ is provided by the Radon-Nikod\'{y}m theorem~\ref{lem:Radon-Nikodym} for a measure 
 $\nu(B) = \int_{Z^{-1}(B)} X(\omega) \P(\D{\omega})$ for $B\in\sB_\sY$
 which is absolutely continuous w.r.t.\ a measure $\mu=\P_Z$. Indeed, $\P_Z(B)=0$ for $B\in\sB$ implies
 \begin{align*}
 \nu(B)=\int_{Z^{-1}(B)} X(\omega)\P(\D{\omega}) = \int_{Z^{-1}(B)} X(\omega)\P_Z[Z(\D{\omega})] = 0.
 \end{align*}
 The required formula \eqref{eq:CE_Y} for scalar valued $X$ is given by the change of variable formula in integral in Lemma~\ref{lem:change_of_var}, i.e.
 $\nu(B) = \int_{B} \CE{X}{Z}(y) \P_Z(\D{y}) = \int_{Z^{-1}(B)} \CE{X}{Z}\circ Z(\omega) \P(\D{\omega})$.
 For a vector valued random variable $\tilde{X}$, the conditional expectation exists for the random variable $X(\omega)=\scal{\tilde{X}(\omega)}{v}_{\sX}$, where $v$ is an arbitrary element from $\sX$.
 The rest can be shown using linearity of integrals and of the conditional expectation.
\end{proof}
\begin{lemma}
 The conditional expectation is equivalent to the problem:\\Find $\CE{X}{Z}\in \sL^1(\sY,\sB,\P_Z;\sX)$ such that the following orthogonality condition holds
 \begin{align}
 \label{eq:orthogonality_dual}
 \dual{X-\CE{X}{Z}\circ Z}{W\circ Z}_{\sL^1\times \sL^\infty} &= 0
 \quad\forall W\in \sL^\infty(\sY,\sB_\sY,\P_Z;\sX).
 \end{align}
\end{lemma}
\begin{proof}Starting from the definition of CE \eqref{eq:CE_Y}, the integrals over the domain $A\in\sigma(Z)$ can be rewritten with the help of characteristic function
 $\ch{A}(\omega) = \ch{Z(A)}\circ Z(\omega)$ as
 \begin{align*}
 \int_{\Omega} [X(\omega)-\CE{X}{Z}\circ Z(\omega)] \ch{Z(A)}\circ Z(\omega) \P(\D{\omega}) &= \V{0} \in\sX
 \end{align*}
 Since the values of the integral are in $\sX$, the scalar product has to be used with the help of an arbitrary vector $v\in\sX$ to obtain
 \begin{align*}
 \int_{\Omega} [X(\omega)-\CE{X}{Z}\circ Z(\omega)] \cdot v\ch{Z(A)}\circ Z(\omega) \P(\D{\omega}) = \dual{X-\CE{X}{Z}\circ Z}{v\ch{A}\circ Z}_{\sL^1\times \sL^\infty} &= 0
 \end{align*}
 the duality with the function $v\ch{A}\circ Z \in \sL^\infty(\Omega,\sigma(Z), \P;\sX)$. The final orthogonality condition \eqref{eq:orthogonality_dual} can be obtained by linearity and density of simple functions in $\sL^\infty(\Omega,\sB,\P_Z;\sX)$.
\end{proof}
\begin{remark}
 For $X\in \sL^2(\Omega,\F{S},\P;\sX)$ the conditional expectation $\CE{X}{Z}$ is from a smaller space $\sL^2(\sY,\sB_\sY,\P_Z;\sX)$ and the orthogonality condition changes from the duality into the scalar product
 \begin{align*}
 \scal{X-\CE{X}{Z}\circ Z}{W\circ Z}_{\sL^2} &= 0
 \quad\forall W\in \sL^2(\sY,\sB_\sY,\P_Z;\sX).
 \end{align*}
 Moreover the conditional expectation can be formulated variationally as a minimisation problem
 \begin{align*}
 \CE{X}{Z} &= \argmin_{W\in \sL^2(\sY,\sB,\P_Z;\sX)} \|X-W\circ Z\|_{\sL^2(\Omega,\F{S},\P;\sX)}^2.
 \end{align*}
\end{remark}

\subsection{Conditional probability and distribution}
\label{sec:conditional-probability-and-distribution}

Here we discuss the connection between conditional expectation, conditional probability, and its density. In the previous section we introduced the conditional expectation $\CE{X}{Z}$ of a random variable $X$ w.r.t.\ a measurement operator $Z$ as an element of $\sL^1(\sY,\sB_\sY,\P_Z;\sX)$. 
Let us first have a look on the conditional expectation of a characteristic function $\ch{M}$ of a set $M\in\F{S}$, which satisfies
\begin{align}
\label{eq:CE_prob}
\int_A \CE{\ch{M}}{Z}\circ Z(\omega)\P(\D{\omega}) &\stackrel{\eqref{eq:CE_Y}}{=} \int_A \ch{M}(\omega)\P(\D{\omega}) =\P(A\cap M)
\quad\forall A\in\sigma(Z), M\in\F{S}.
\end{align}

Similar to the expectation of characteristic functions $\E{\ch{M}} = \P(M)$ that equals to probability, the \emph{conditional probability function} is defined as a conditional expectation of a corresponding characteristic function
\begin{align*}
\P[M|Z] &= \CE{\ch{M}}{Z} \in \sL^1(\sY,\sB,\P_Z;\sR)
\end{align*}
i.e.\ a class of equivalent functions from $\sL^1(\sY,\sB,\P_Z;\sR)$. 
A suitable element of this class is called  a \emph{regular conditional probability} if
\begin{enumerate}
 \item $\P[M|Z]\circ Z(\cdot):\Omega\rightarrow[0,1]$ is a measurable function for each $M\in\F{S}$,
 \item $\P[\cdot|Z](y):\F{S}\rightarrow[0,1]$ is a probability measure for $y\in Z[\Omega\setminus \Gamma]$ such that $\P(\Gamma)=0$.
\end{enumerate}
Unfortunately, the conditional probability function $\P[M|Z]$ may fail to be a regular conditional probability \cite[Proposition~3.2.2]{Rao2006book}.
Better situation occurs for a \emph{conditional probability distribution} as a push-forward (an image) conditional probability
\begin{align}
\label{eq:CP_dist}
\P_\Q[A|Z] &= \P[\Q^{-1}(A)|Z] = \CE{\ch{A}\circ \Q}{Z} \quad\forall A\in\sB_\sQ
\end{align}
which always exists with properties (i) and (ii) when $\sQ$ (range of $\Q$) is a finite dimensional space \cite[Theorem~3.2.5]{Rao2006book}. It means that $\P_\Q[\cdot|Z](\cy):\sB_\sQ\rightarrow[0,1]$ is the conditional probability measure on the parameter space $\sQ$.

\begin{proposition}[Connection of conditional probabilities to probability densities]
 \label{lem:CP_dens}
 Let $\Q\in \sL^1(\Omega,\F{S},\P;\sQ)$ be a parameter RV and $Z\in \sL^1(\Omega,\F{S},\P;\sY)$ be a measurement RV, and $\P_{\Q,Z}$ be a probability measure on $\sB_\sQ\times\sB_\sY$ defined as $\P_{\Q,Z}(A\times B) = \P(\Q^{-1}(A)\cap Z^{-1}(B))$.
 Let $\P_\Q[\cdot|Z]$ and $\P_\Q$ on $\sB_\sQ$ be the conditional distributional measure and push-forward measure of $\Q$, respectively.
 Assuming that there exist probability densities (Radon-Nikod\'{y}m theorem~\eqref{lem:Radon-Nikodym}) w.r.t.\ Lebesgue measure on $\sQ$, $\sY$, and $\sQ\times\sY$
 \begin{subequations}
  \label{eq:prob_dens}
  \begin{align}
  \P_\Q[A|Z](y) &= \int_A f_{\Q|Z}(\q|y)\D{\q}\quad\forall A\in\sB_\sQ, y \in \sY,
  \\
  \P_\Q(A) &= \int_A f_{\Q}(\q)\D{\q}\quad\forall A\in\sB_\sQ,
  \\
  \P_Z(B) &= \int_B f_{Z}(y)\D{y}\quad\forall B\in\sB_\sY,
  \\
  \P_{\Q,Z}(A\times B) &= \int_{A\times B} f_{\Q,Z}(\q,y)\D{\q}\D{y}\quad\forall A\in\sB_\sQ \text{ and } B\in\sB_\sY,
  \end{align}
 \end{subequations}
 then the conditional probabilities can be expressed in terms of probability densities
 \begin{align*}
 f_{\Q|Z}(\q|y) = \frac{f_{Z|\Q}(y|\q)f_\Q(\q)}{\int_\sQ f_{Z|\Q}(y|\q)f_\Q(\q)\D{\q}}.
 \end{align*}
\end{proposition}

\begin{proof}
 The proof that is provided in \cite[Section~3.2, Proposition~7]{Rao2006book} for cumulative distribution functions is formulated here in terms of conditional probabilities.
 
 The definition of the joint probability $\P_{\Q\times Z}$ on $\sQ\times\sY$ can be rewritten using conditional expectation w.r.t.\ $Z$, for $A\in\sB_\sQ$ and $B\in\sB_\sY$
 \begin{align*}
 \P_{\Q,Z}(A \times B) &= \P(\Q^{-1}(A)\cap Z^{-1}(B))
 \\
 & \stackrel{\eqref{eq:CE_prob}}{=} \int_{Z^{-1}(B)} \CE{\ch{A}\circ \Q}{Z}\circ Z(\omega)\P(\D{\omega})
 \stackrel{\eqref{eq:change_of_var}}{=}
 \int_{B} \CE{\ch{A}\circ \Q}{Z}(y)\P_Z(\D{y})
 \end{align*}
 Using the conditional distribution \eqref{eq:CP_dist} and probability densities \eqref{eq:prob_dens}, one can further deduce
 \begin{align*}
 \int_{B} \CE{\ch{A}\circ \Q}{Z}(y)\P_Z(\D{y}) &=
 \int_{B} \P_\Q[A|Z](y)\P_Z(\D{y})
 = \int_{B} \P_\Q[A|Z](y)f_Z(y)\D{y}
 \\
 &
 = \int_{A\times B} f_{\Q|Z}(\q|y)f_Z(y)\D{\q}\D{y}.
 \end{align*}
 
 Using probability density also for $\P_{\Q,Z}$, we obtain the relation in PDF
 \begin{align*}
 \int_{A\times B} f_{\Q,Z}(\q,y) \D{\q}\D{y} &= \int_{A\times B} f_{\Q|Z}(\q|y)f_Z(y)\D{\q}\D{y}.
 \end{align*}
 for arbitrary sets $A\in\sB_\sQ$ and $B\in\sB_\sY$. Therefore, the similar relation holds for the arguments
 \begin{align*}
 f_{\Q,Z}(\q,y) &= f_{\Q|Z}(\q|y)f_Z(y)
 \quad \text{for almost all }(\q,y)\in\sQ\times\sY.
 \end{align*}
 Analogical result can be obtained with a change of random variables
 $$f_{\Q,Z}(\q,y) = f_{Z|\Q}(y|\q)f_\Q(\q),$$ from which the conditional probability density can be deduced
 \begin{align*}
 f_{\Q|Z}(\q|y) = \frac{f_{Z|\Q}(y|\q)f_\Q(\q)}{f_Z(\q)}.
 \end{align*}
 To finish the proof, we will show that $\int_{\sQ} f_{Z|\Q}(y|\q)f_\Q(\q)\D{\q} = f_Z(y)$. Indeed
 \begin{align*}
 \int_{\sQ\times B} f_{Z|\Q}(y|\q)f_\Q(\q)\D{\q}\D{y} &= \int_{\sQ\times B} f_{\Q,Z}(\q,y)\D{\q} = \P_{\Q,Z}(\sQ\times B) =\P(\Q^{-1}(\sQ)\cap Z^{-1}(B))
 \\
 &=\P(Z^{-1}(B)) 
 = \P_Z(B) = \int_Bf_Z(y)\D{y},
 \end{align*}
 which is valid for an arbitrary set $B$.
\end{proof}
Next lemma is specialised for a measurement with additive error which is described in more detail in Section~\ref{sec:model-problem}.
\begin{lemma}
 \label{lem:likelihood}
 Let the measurement be given with an additive error \eqref{eq:measurement}, i.e.\ $Z(\omega) = \Yq(\Q(\xi))+E(\theta)$ for $\omega=(\xi,\theta)\in \Xi\times\Theta=\Omega$ and parameters $\Q$ depends only on values $\xi\in\Xi$ .
 Then the likelihood is equal almost surely to
 \begin{align*}
 f_{Z|Q}(y|q) = f_E(y-\Yq(q)).
 \end{align*}
\end{lemma}
\begin{proof}
 We start from the conditional probability distribution
 \begin{align*}
 \P_{\YE}(B|Q)(q) &= \CE{\ch{B}\circ Z}{Q}(q)\quad\text{for }B\in\sB_\sY,
 \end{align*}
 which is defined as
 \begin{align*}
 \int_{Q^{-1}(A)} \CE{\ch{B}\circ Z}{Q}\circ Q (\omega)\P(\D{\omega}) &= \int_{Q^{-1}(A)} \ch{B}\circ Z(\omega) \P(\D{\omega})
 \quad\forall A\in \sB_\sQ.
 \end{align*}
 In fact the integral over the intersection of preimages $Q^{-1}(A)\cap \YE^{-1}(B)$ is needed. We will use that $Q$ does not depend on the whole domain $\Omega=\Xi\times\Theta$ but only on $\xi\in\Xi$ and error $E$ only on $\theta\in\Theta$. Particularly the preimages can be expressed as
 \begin{align*}
 Q^{-1}(A) &= \{(\xi,\theta) | Q(\xi)\in A \wedge \theta\in\Theta\}
 \\
 \YE^{-1}(B) &= \{(\xi,\theta) | Y_Q(Q(\xi))+E(\theta) \in B \} = \{(\xi,\theta) | \xi\in\Xi \wedge E(\theta) \in B - Y_Q(Q(\xi)) \}
 \end{align*}
 which allows as to express their intersection as
 \begin{align*}
 Q^{-1}(A)\cap \YE^{-1}(B) &= \{ (\xi,\theta) \in \Omega=\Xi\times\Theta \sep  Q(\xi)\in A \wedge E(\theta) \in B - Y_Q(Q(\xi)) \}
 .
 \end{align*}
 With the help of the change of variable formula in the integral (Lemma~\ref{lem:change_of_var}), the conditional probability distribution is recalculated with maps $\Q$ and $E$ as
 \begin{align*}
 \int_{Q^{-1}(A)} \ch{B}\circ Z(\omega) \P(\D{\omega}) &= \int_{Q^{-1}(A)\cap \YE^{-1}(B)} \P_\Xi(\D{\xi})\times \P_\Theta(\D{\theta})
 \\
 &= \int_{A} \left( \int_{B-Y_Q(q)} \P_E(\D{y})\right)\P_\Q(\D{\q})
 \\
 &= \int_{A} \left( \int_{B} f_E(y-Y_Q(q)) \D{y} \right) \P_\Q(\D{\q})
 \\
 &= \int_{A\times B} f_E(y-Y_Q(q)) f_\Q(q) \D{\q} \D{y}.
 \end{align*}
 Following the proof of the previous Proposition~\ref{lem:CP_dens}, the conditional expectation also equals
 \begin{align*}
 \int_{\Q^{-1}(A)} \ch{B}\circ Z(\omega) \P(\D{\omega}) = \int_{A\times B} f_{\YE|Q}(y|q) f_\Q(q) \D{\q} \D{y}
 \end{align*}
 from which the required equality comes out as it holds for arbitrary sets $A\in\sB_\sQ$ and $B\in\sB_\sY$.
\end{proof}
\section{Auxiliary lemmata}
\begin{theorem}[Radon-Nikod\'{y}m theorem]
 \label{lem:Radon-Nikodym}
 Let $\nu$ and $\mu$ be two probability measures on $(\Omega,\F{S})$ such that $\mu$ is absolutely continuous w.r.t $\nu$ ($\nu\ll\mu$ if $\nu(A)=0$ for all $A\in\F{S}$ with $\mu(A)=0$)
 Then there exist a unique non-negative function (called probability density) $f\in \sL^1(\Omega,\F{S},\mu;\sR)$ such that
 \begin{align*}
 \nu(A) = \int_{A}f(\omega)\mu(\D{\omega}).
 \end{align*}
\end{theorem}
\begin{proof}
 See in \cite[Theorem~6.9]{Rudin1986real}.
\end{proof}
\begin{lemma}[Change of the variable formula in the integral]
 \label{lem:change_of_var}
 Let $Y\in\sL^1(\Omega,\F{S},\P;\sY)$ and $X\in\sL^1(\sY,\sB,\P_Y;\sX)$ be two measurable functions and $\P_Y:\sB\rightarrow[0,1]$ be a push-forward measure defined as $\P_Y(B):=\P(Y^{-1}(B))$ for $B\in\sB$.
 Then for all $A\in\F{S}$ the change of variable formula holds
 \begin{align}
 \label{eq:change_of_var}
 \int_{A} X\circ Y(\omega)\P(\D{\omega}) = \int_{Y(A)} X(y)\P_Y(\D{y}).
 \end{align}
 if either side exists.
\end{lemma}
\begin{proof}See in \cite[Section~1.4, Theorem~1]{Rao2006book}.
\end{proof}

\section*{Acknowledgement}
Funded by the Deutsche Forschungsgemeinschaft (DFG, German Research Foundation) - project numbers MA2236/21-1; MA2236/26-1; and MA2236/30-1.


\end{document}